\documentclass[12pt, reqno]{amsart}
\pdfoutput=1

\usepackage[margin=2cm]{geometry}
\usepackage{amsmath, amssymb, amsfonts, amstext, verbatim, amsthm, mathrsfs}
\usepackage{microtype}
\usepackage[all,arc,knot,poly,cmtip]{xy}
\usepackage{mathtools}
\usepackage[colorlinks=true,linkcolor=blue,citecolor=blue,urlcolor=blue,citebordercolor={0 0 1},urlbordercolor={0 0 1},linkbordercolor={0 0 1}]{hyperref} 
\usepackage{enumerate}
\usepackage{stmaryrd}
\usepackage{subcaption}
\usepackage{tikz}
\usepackage{tikz-cd}
\usepackage{setspace}
\usepackage{graphicx}
\theoremstyle{plain}
\newtheorem{theorem}{Theorem}[section]
\newtheorem{proposition}[theorem]{Proposition}
\newtheorem{lemma}[theorem]{Lemma}

\theoremstyle{definition}
\newtheorem{definition}[theorem]{Definition}

\DeclareMathOperator{\Hom}{\operatorname{Hom}}
\DeclareMathOperator{\Tr}{\operatorname{Tr}}
\DeclareMathOperator{\Sk}{\operatorname{Sk}}
\DeclareMathOperator{\Gr}{\operatorname{Gr}}
\DeclareMathOperator{\diag}{\operatorname{diag}}
\DeclareMathOperator{\End}{\operatorname{End}}
\DeclareMathOperator{\Mod}{\operatorname{Mod}}

\begin{document}

\title[Dualities of $K$-theoretic Coulomb branches]{Dualities of $K$-theoretic Coulomb branches \protect\\[0.4\baselineskip] from a once-punctured torus}

\author{Dylan G.L. Allegretti}
\author{Peng Shan}
\address{Yau Mathematical Sciences Center, Department of Mathematical Sciences, Tsinghua~University, 100084, Beijing, China}
\email{dylanallegretti@tsinghua.edu.cn}
\email{pengshan@tsinghua.edu.cn}

\date{}

\maketitle

\begin{abstract}
We consider the quantized $\mathrm{SL}_2$-character variety of a once-punctured torus. We show that this quantized algebra has three $\mathbb{Z}_2$-invariant subalgebras that are isomorphic to quantized $K$-theoretic Coulomb branches in the sense of Braverman, Finkelberg, and Nakajima. These subalgebras are permuted by the $\mathrm{SL}_2(\mathbb{Z})$ mapping class group action. Our results confirm various predictions from the physics literature about 4d~$\mathcal{N}=2^*$ theories and their dualities.
\end{abstract}

\setcounter{tocdepth}{1}
\tableofcontents

\section{Introduction}

Quantum field theory has long served as a source of new ideas and conjectures in mathematics. In the past decade in particular, the work of Gaiotto, Moore, and Neitzke~\cite{GMN13a} has inspired a number of exciting mathematical developments. Building on the seminal work of Gaiotto~\cite{G12}, these authors studied a class of four-dimensional quantum field theories, called $\mathcal{N}=2$ theories of class~S, whose properties are encoded in the geometry of a Riemann surface. Their work points to fascinating connections between Hitchin systems, cluster algebras, and the wall-crossing properties of Donaldson--Thomas invariants.

Associated to any $\mathcal{N}=2$ field theory is a space called the Coulomb branch of the theory on $\mathbb{R}^3\times S^1$. This space was given a precise mathematical definition in the work of Braverman, Finkelberg, and Nakajima~\cite{BFN18}. An interesting prediction of Gaiotto, Moore, and Neitzke is that, in the case of a class~S theory, this Coulomb branch should be related to a character variety of the surface that defines the theory. In previous work~\cite{AS24}, we considered the quantized $\mathrm{SL}_2$-character variety of a punctured surface. We formulated a precise relationship between this quantized character variety and a quantized Coulomb branch, confirming the prediction of Gaiotto, Moore, and Neitzke in some cases.

In the present paper, we consider the quantized $\mathrm{SL}_2$-character variety of a once-punctured torus. We describe three $\mathbb{Z}_2$-invariant subalgebras of this quantized algebra. Building on our earlier results~\cite{AS24}, we show that each of these subalgebras is isomorphic to a quantized Coulomb branch in the sense of Braverman, Finkelberg, and Nakajima~\cite{BFN18}. They are permuted by the $\mathrm{SL}_2(\mathbb{Z})$ mapping class group action. In physical language, these subalgebras are quantized Coulomb branches of the 4d~$\mathcal{N}=2^*$ theories from~\cite{DW96}. Our results confirm various predictions from the physics literature about these theories and their dualities~\cite{AST13,GMN13b,GKNPS23}.

\subsection{The quantized $\mathrm{SL}_2$-character variety of~$S_{1,1}$}

Let $S=S_{g,n}$ be the surface obtained from a closed surface~$\bar{S}$ of genus~$g$ by removing $n$ distinct points $\{p_1,\dots,p_n\}\subset\bar{S}$. The $\mathrm{SL}_2$-character variety of~$S$ is a moduli space parametrizing principal $\mathrm{SL}_2$-bundles with flat connection on~$S$. Concretely, it can be defined by considering representations $\pi_1(S)\rightarrow\mathrm{SL}_2$ of the fundamental group of~$S$ into $\mathrm{SL}_2$. The set $\Hom(\pi_1(S),\mathrm{SL}_2)$ of such representations has the structure of an affine variety, and the $\mathrm{SL}_2$-character variety of $S$ is the affine GIT~quotient 
\[
\mathcal{M}_{\text{flat}}(S,\mathrm{SL}_2)=\Hom(\pi_1(S),\mathrm{SL}_2)\sslash\mathrm{SL}_2
\]
of this variety by the conjugation action of~$\mathrm{SL}_2$. Note that any closed loop $\gamma\subset S$ defines a regular function $\Tr_\gamma$ on the character variety by the rule $\Tr_\gamma(\rho)=\Tr\rho(\gamma)$.

We will be interested in the subspace of $\mathcal{M}_{\text{flat}}(S,\mathrm{SL}_2)$ obtained by prescribing the conjugacy class of the monodromy of the flat connection around each puncture. More precisely, let $\gamma_1,\dots,\gamma_n\subset S$ be simple closed curves where $\gamma_i$ surrounds the point $p_i\in\bar{S}$, and let $\lambda=(\lambda_1,\dots,\lambda_n)\in(\mathbb{C}^*)^n$. Then the relative character variety $\mathcal{M}_{\text{flat}}^\lambda(S,\mathrm{SL}_2)\subset\mathcal{M}_{\text{flat}}(S,\mathrm{SL}_2)$ is the closed subscheme of the character variety cut out by the equations 
\[
\Tr_{\gamma_i}=\lambda_i+\lambda_i^{-1}
\]
for $i=1,\dots,n$.

We will consider the special case where $S=S_{1,1}$ is a once-punctured torus. Let $\alpha$,~$\beta$ be two loops in~$S_{1,1}$ generating the fundamental group $\pi_1(S_{1,1})\cong\mathbb{Z}*\mathbb{Z}$. A classical result going back to Fricke~\cite{F96} and Vogt~\cite{V89} (see also~\cite{G09}) says that, for $\lambda\in\mathbb{C}^*$, the relative $\mathrm{SL}_2$-character variety $\mathcal{M}_{\text{flat}}^\lambda(S_{1,1},\mathrm{SL}_2)$ is an affine cubic surface 
\[
\mathcal{M}_{\text{flat}}^\lambda(S_{1,1},\mathrm{SL}_2(\mathbb{C}))\cong\{(a,b,c)\in\mathbb{C}^3:a^2+b^2+c^2+abc=2+\lambda+\lambda^{-1}\}
\]
where the isomorphism is given by 
\[
\rho\mapsto(-\Tr\rho(\alpha),-\Tr\rho(\beta),-\Tr\rho(\alpha\beta)).
\]
In particular, this result implies that $\Tr_\alpha$, $\Tr_\beta$, and $\Tr_{\alpha\beta}$ generate the full algebra of regular functions on the relative character variety.

Associated to any surface $S=S_{g,n}$ is a noncommutative algebra $\Sk_A(S)$, called the Kauffman~bracket skein algebra, which is a quantization of the $\mathrm{SL}_2$-character variety of~$S$. As we will explain below, it is a $\mathbb{C}[A^{\pm1}]$-algebra generated by isotopy classes of framed links in the three-manifold $S\times[0,1]$, and its specialization $\Sk_A(S)\otimes_{\mathbb{C}[A^{\pm1}]}(\mathbb{C}[A^{\pm1}]/(A+1))$ at $A=-1$ is identified with the algebra of regular functions on~$\mathcal{M}_{\text{flat}}(S,\mathrm{SL}_2)$. Following~\cite{AS24}, we define a variant of this algebra, called the relative skein algebra and denoted $\Sk_{A,\lambda}(S)$. It is a noncommutative $\mathbb{C}[A^{\pm1},\lambda_1^{\pm1},\dots,\lambda_n^{\pm1}]$-algebra such that we recover the algebra of regular functions on~$\mathcal{M}_{\text{flat}}^\lambda(S,\mathrm{SL}_2)$ by specializing $A$ to~$-1$ and specializing the variables $\lambda_i$ to nonzero complex numbers.

When $S=S_{1,1}$ is a once-punctured torus, we can give an explicit description of the relative skein algebra $\Sk_{A,\lambda}(S_{1,1})$. Namely, it is the $\mathbb{C}[A^{\pm1},\lambda^{\pm1}]$-algebra generated by elements $\alpha$, $\beta$, and~$\gamma$, satisfying the relations 
\[
[\alpha,\beta]_A=(A^{-2}-A^2)\gamma, \quad [\beta,\gamma]_A=(A^{-2}-A^2)\alpha, \quad [\gamma,\alpha]_A=(A^{-2}-A^2)\beta,
\]
where $[x,y]_A\coloneqq A^{-1}xy-Ayx$, and the relation 
\[
A^{-2}\alpha^2+A^2\beta^2+A^{-2}\gamma^2-A^{-1}\alpha\beta\gamma=A^2+A^{-2}+\lambda+\lambda^{-1}.
\]
As we review below, this algebra is closely related to the spherical double affine Hecke algebra of type~$A_1$~\cite{C05}. Note that when we specialize to $A=-1$, the first three relations express the commutativity of the elements $\alpha$, $\beta$, and~$\gamma$, while the last relation becomes the defining equation of the cubic surface $\mathcal{M}_{\text{flat}}^\lambda(S_{1,1},\mathrm{SL}_2)$.

\subsection{Quantized $K$-theoretic Coulomb branches}

To define the Coulomb branch mathematically following the approach of Braverman, Finkelberg, and Nakajima~\cite{BFN18}, we start with a group~$G$ called the gauge group and a representation $N$ of~$G$. We assume the action of~$G$ on~$N$ extends to an action of $\widetilde{G}=G\times F$ for some group $F$ called a flavor symmetry group. For the theories studied in this paper, we take $N=\Hom_{\mathbb{C}}(\mathbb{C}^2,\mathbb{C}^2)$ to be the space of complex $2\times2$~matrices, we let $G=\mathrm{SL}_2$ or~$\mathrm{PGL}_2$ act on~$N$ by conjugation, and we let $F=\mathbb{C}^*$ act on~$N$ by scalar multiplication.

For general $G$ and~$N$, Braverman, Finkelberg, and Nakajima define a remarkable geometric object $\mathcal{R}_{G,N}$ called the variety of triples. It is an ind-scheme generalizing the affine Grassmannian of~$G$. The variety of triples is equipped with a natural action of $\widetilde{G}_{\mathcal{O}}$, the group of $\mathcal{O}$-valued points of~$\widetilde{G}$ where $\mathcal{O}=\mathbb{C}\llbracket z\rrbracket$. There is a further action of~$\mathbb{C}^*$ on~$\mathcal{R}_{G,N}$ by loop rotation. Using standard methods of geometric representation theory~\cite{BFN18,CG97,VV10}, one can define the $\widetilde{G}_{\mathcal{O}}$-equivariant $K$-theory $K^{\widetilde{G}_{\mathcal{O}}}(\mathcal{R}_{G,N})$ and the $\widetilde{G}_{\mathcal{O}}\rtimes\mathbb{C}^*$-equivariant $K$-theory $K^{\widetilde{G}_{\mathcal{O}}\rtimes\mathbb{C}^*}(\mathcal{R}_{G,N})$ of the variety of triples.

Braverman, Finkelberg, and Nakajima showed in~\cite{BFN18} that $K^{\widetilde{G}_{\mathcal{O}}}(\mathcal{R}_{G,N})$ is equipped with a natural convolution product, making it into a commutative $\mathbb{C}$-algebra, and they defined the $K$-theoretic Coulomb branch to be the spectrum of this algebra. Braverman, Finkelberg, and Nakajima also showed that $K^{\widetilde{G}_{\mathcal{O}}\rtimes\mathbb{C}^*}(\mathcal{R}_{G,N})$ has a convolution product, making it into a noncommutative algebra over $K^{\mathbb{C}^*}(\mathrm{pt})\cong\mathbb{C}[q^{\pm1}]$. The variable~$q$ can be seen as a quantization parameter, and one recovers the commutative algebra $K^{\widetilde{G}_{\mathcal{O}}}(\mathcal{R}_{G,N})$ by specializing to $q=1$. In the following, we will replace the group $\mathbb{C}^*$ by its double cover and thereby extend scalars to $\mathbb{C}[q^{\pm\frac{1}{2}}]$.

\subsection{Statement of the main result}
\label{sec:StatementOfTheMainResult}

To state the main result of this paper, we consider three involutions $\xi_1$, $\xi_2$, $\xi_3$ of the relative skein algebra $\Sk_{A,\lambda}(S_{1,1})$. In terms of the presentation given above, we have 
\begin{alignat*}{4}
\xi_1&: \quad &&\alpha\mapsto-\alpha, \quad &&\beta\mapsto\beta, \quad &&\gamma\mapsto-\gamma, \\
\xi_2&: \quad &&\alpha\mapsto\alpha, \quad &&\beta\mapsto-\beta, \quad &&\gamma\mapsto-\gamma, \\
\xi_3&: \quad &&\alpha\mapsto-\alpha, \quad &&\beta\mapsto-\beta, \quad &&\gamma\mapsto\gamma.
\end{alignat*}
The involution $\xi_2$ appeared in our previous paper~\cite{AS24}. There we showed that the $\xi_2$-invariant subalgebra of $\Sk_{A,\lambda}(S_{1,1})$ is isomorphic to a quantized $K$-theoretic Coulomb branch associated to the group $\mathrm{SL}_2$ and representation $N=\Hom_{\mathbb{C}}(\mathbb{C}^2,\mathbb{C}^2)$.

\begin{theorem}[\cite{AS24}, Theorem~1.3]
There is a $\mathbb{C}$-algebra isomorphism 
\[
\Sk_{A,\lambda}(S_{1,1})^{\xi_2}\cong K^{(\mathrm{SL}_2\times\mathbb{C}^*)_{\mathcal{O}}\rtimes\mathbb{C}^*}(\mathcal{R}_{\mathrm{SL}_2,N}).
\]
\end{theorem}

In the present paper, we consider the invariant subalgebras corresponding to the involutions~$\xi_1$ and~$\xi_3$ and show that these are naturally isomorphic to quantized Coulomb branches for the Langlands dual gauge group~$\mathrm{PGL}_2$.

\begin{theorem}
\label{thm:intromain}
There are $\mathbb{C}$-algebra isomorphisms 
\[
\Sk_{A,\lambda}(S_{1,1})^{\xi_1}\cong K^{(\mathrm{PGL}_2\times\mathbb{C}^*)_{\mathcal{O}}\rtimes\mathbb{C}^*}(\mathcal{R}_{\mathrm{PGL}_2,N})\cong\Sk_{A,\lambda}(S_{1,1})^{\xi_3}.
\]
\end{theorem}

In the above theorems, the isomorphisms map the quantization parameter $A$ in the relative skein algebra to the element $q^{-\frac{1}{2}}$ in the quantized Coulomb branches. We can therefore take the classical limit and relate each of the $\mathbb{Z}_2$-quotients $\mathcal{M}_{\text{flat}}^\lambda(S_{1,1},\mathrm{SL}_2)\sslash\langle\xi_i\rangle$ to a $K$-theoretic Coulomb branch. This confirms the predictions of Section~8.4 of~\cite{GMN13b} and Section~4.2 of~\cite{GKNPS23}. Note that our gauge groups $\mathrm{SL}_2$ and~$\mathrm{PGL}_2$ are the complexifications of the groups $\mathrm{SU}(2)$ and $\mathrm{SO}(3)$ appearing in these physics references, consistent with the requirements of~\cite{BFN18}.

Finally, let us consider the mapping class group $\mathrm{SL}_2(\mathbb{Z})\cong\langle\sigma,\tau_+|\sigma^4=1,(\sigma\tau_+)^3=\sigma^2\rangle$ of~$S_{1,1}$. This group acts by automorphisms on the relative skein algebra $\Sk_{A,\lambda}(S_{1,1})$, and the action is compatible with the action on spherical DAHA (see Proposition~\ref{prop:SL2Zequivariance}). If we set $\tau_-=\tau_+\sigma^{-1}\tau_+$, then the elements $\tau_+$,~$\tau_-$, and~$\sigma$ permute the $\xi_i$-invariant subalgebras of $\Sk_{A,\lambda}(S_{1,1})$ in the manner illustrated in Figure~\ref{fig:skeindualities}. Physically, $\mathrm{SL}_2(\mathbb{Z})$ is the S-duality group of the $\mathcal{N}=2^*$ theories, and the symmetries in Figure~\ref{fig:skeindualities} follow from the dualities depicted in the diagram~(4.2) of~\cite{GKNPS23}.

\begin{figure}[ht]
\begin{tikzpicture}[scale=1.25]
\node(N) at (0,1.732) {$\Sk_{A,\lambda}(S_{1,1})^{\xi_2}$};
\node(N1) at (N.north) {};
\node(N2) at (N.south) {};
\node(E) at (2,0) {$\Sk_{A,\lambda}(S_{1,1})^{\xi_3}$};
\node(E1) at (E.east) {};
\node(W) at (-2,0) {$\Sk_{A,\lambda}(S_{1,1})^{\xi_1}$};
\node(W1) at (W.west) {};
\path[->]  (E1.north) edge [loop right,out=30,in=-30,distance=2.5em] node [right] {\tiny $\sigma$} (E1.south);
\path[->]  (W1.south) edge [loop right,out=210,in=150,distance=2.5em] node [left] {\tiny $\tau_-$} (W1.north);
\path[->]  (N1.west) edge [loop right,out=120,in=60,distance=2.5em] node [above] {\tiny $\tau_+$} (N1.east);
\draw[<->]  (W.east) -- node [midway,below] {\tiny $\tau_+$} (E.west);
\draw[<->]  (W.north) -- node [midway,above] {\tiny $\sigma$} (N2.west);
\draw[<->]  (E.north) -- node [midway,above] {\tiny $\tau_-$} (N2.east);
\end{tikzpicture}
\caption{Invariant subalgebras of $\Sk_{A,\lambda}(S_{1,1})$.\label{fig:skeindualities}}
\end{figure}
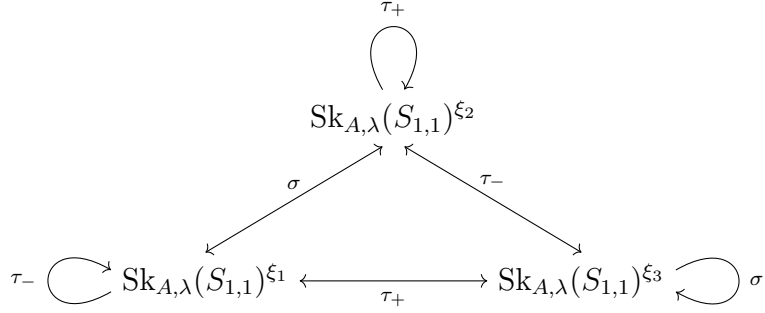

\subsection*{Acknowledgements.}
The authors thank Du~Pei for helpful discussions related to this work. PS is supported by NSFC Grant No.~12225108 and by the New Cornerstone Science Foundation through the Xplorer Prize.

\section{The quantized $\mathrm{SL}_2$-character variety of $S_{1,1}$}

In this section, we give a presentation of the relative skein algebra of a once-punctured torus, construct a faithful representation of this algebra, and describe its $\mathrm{SL}_2(\mathbb{Z})$ symmetries.

\subsection{A presentation of the skein algebra}

Let $S=S_{g,n}$ be the surface obtained from a closed surface~$\bar{S}$ of genus~$g$ by removing $n$ distinct points $\{p_1,\dots,p_n\}\subset\bar{S}$. Following Section~2 of~\cite{AS24}, we write $\mathcal{L}_A(S)$ for the $\mathbb{C}[A^{\pm1}]$-module generated by isotopy classes of framed links in~$S\times[0,1]$. If $[L_1]$,~$[L_2]\in\mathcal{L}_A(S)$ are the isotopy classes of framed links $L_1$,~$L_2\subset S\times[0,1]$, then we can find isotopic links $L_1'\simeq L_1$ and $L_2'\simeq L_2$ that lie in $S\times[0,\frac{1}{2})$ and $S\times(\frac{1}{2},1]$, respectively. We define the product $[L_1][L_2]\in\mathcal{L}_A(S)$ as the isotopy class of the union $L_1'\cup L_2'\subset S\times[0,1]$. Extending $\mathbb{C}[A^{\pm1}]$-bilinearly, we see that $\mathcal{L}_A(S)$ has the natural structure of a $\mathbb{C}[A^{\pm1}]$-algebra.

We say that a framed link $L\subset S\times[0,1]$ has the \emph{blackboard framing} if the framing vector at any point of~$L$ is tangent to the $[0,1]$-factor and points towards~1. Given any framed link $L\subset S\times[0,1]$, we can always find an isotopic framed link $L'\simeq L$ that has the blackboard framing. In the following, we represent the isotopy class of~$L$ by drawing the projection of~$L'$ to~$S$ and, if two points of~$L'$ project to the same point of~$S$, indicating the order of their $[0,1]$-coordinates. We can then define the \emph{(Kauffman bracket) skein algebra}~$\Sk_A(S)$ to be the quotient of~$\mathcal{L}_A(S)$ by the relations illustrated in Figure~\ref{fig:skein}. In each of these relations, the diagrams represent framed links in~$S\times[0,1]$. We show only the part of the link that projects to a small neighborhood in~$S$, shaded in gray, and assume that the links in a given relation are identical outside this neighborhood.

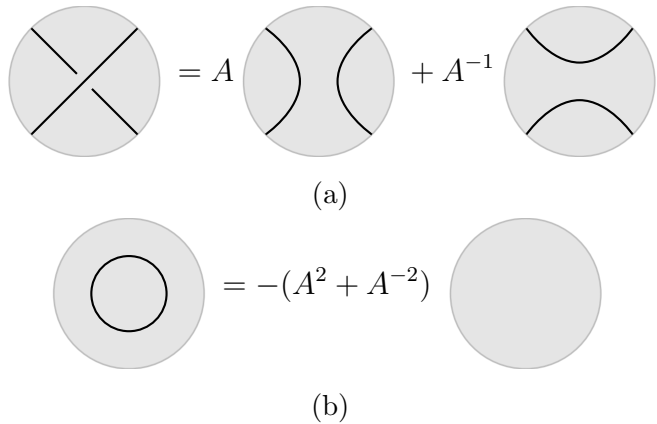
\begin{figure}[ht]
\begin{subfigure}{\textwidth}
\[
\begin{tikzpicture}[baseline=(current bounding box.center)]
\clip(-1.125,-1) rectangle (1.125,1);
\draw[fill=gray,opacity=0.2] (0,0) circle (1);
\draw[black, thick] (-0.707,-0.707) -- (0.707,0.707);
\draw[black, thick] (-0.707,0.707) -- (-0.1,0.1);
\draw[black, thick] (0.1,-0.1) -- (0.707,-0.707);
\draw[gray, opacity=0.2, thick] (0,0) circle (1);
\end{tikzpicture}
= A
\begin{tikzpicture}[baseline=(current bounding box.center)]
\clip(-1.125,-1) rectangle (1.125,1);
\draw[fill=gray,opacity=0.2] (0,0) circle (1);
\draw[black, thick] plot [smooth, tension=1] coordinates { (0.707,0.707) (0.25,0) (0.707,-0.707)};
\draw[black, thick] plot [smooth, tension=1] coordinates { (-0.707,0.707) (-0.25,0) (-0.707,-0.707)};
\draw[gray, opacity=0.2, thick] (0,0) circle (1);
\end{tikzpicture}
+ A^{-1}
\begin{tikzpicture}[baseline=(current bounding box.center)]
\clip(-1.125,-1) rectangle (1.125,1);
\draw[fill=gray,opacity=0.2] (0,0) circle (1);
\draw[black, thick] plot [smooth, tension=1] coordinates { (-0.707,0.707) (0,0.25) (0.707,0.707)};
\draw[black, thick] plot [smooth, tension=1] coordinates { (-0.707,-0.707) (0,-0.25) (0.707,-0.707)};
\draw[gray, opacity=0.2, thick] (0,0) circle (1);
\end{tikzpicture}
\]
\caption{\label{subfig:resolution}}
\end{subfigure}
\begin{subfigure}{\textwidth}
\[
\begin{tikzpicture}[baseline=(current bounding box.center)]
\clip(-1.125,-1) rectangle (1.125,1);
\draw[fill=gray,opacity=0.2] (0,0) circle (1);
\draw[black, thick] (0,0) circle (0.5);
\draw[gray, opacity=0.2, thick] (0,0) circle (1);
\end{tikzpicture}
= -(A^2+A^{-2})
\begin{tikzpicture}[baseline=(current bounding box.center)]
\clip(-1.25,-1) rectangle (1.25,1);
\draw[fill=gray,opacity=0.2] (0,0) circle (1);
\draw[gray, opacity=0.2, thick] (0,0) circle (1);
\end{tikzpicture}
\]
\caption{\label{subfig:unknot}}
\end{subfigure}
\caption{The Kauffman bracket skein relations.\label{fig:skein}}
\end{figure}

In the following, we will be interested in a variant of the skein algebra~$\Sk_A(S)$, defined in~\cite{AS24}. Namely, we define the \emph{relative skein algebra} to be the quotient of $\Sk_A(S)\otimes_{\mathbb{C}[A^{\pm1}]}\mathbb{C}[A^{\pm1},\lambda_1^{\pm1},\dots,\lambda_n^{\pm1}]$ by the additional relation illustrated in Figure~\ref{fig:puncturerelation}. The diagrams in this relation represent framed links in~$S\times[0,1]$. We show only the part of the link that projects to a small neighborhood of the puncture $p_i\in\bar{S}$, shaded in gray, and we assume that the links in this relation are identical outside this neighborhood. As explained in~\cite{AS24}, the algebra defined in this way provides a quantization of the relative $\mathrm{SL}_2$-character variety of~$S$.

\begin{figure}[ht]
\begin{tikzpicture}
\draw[fill=gray,opacity=0.2] (0,0) circle (1);
\draw[black, thick] (0,0) circle (0.5);
\draw[gray, opacity=0.2, thick] (0,0) circle (1);
\node at (0,0) {\tiny $\bullet$};
\end{tikzpicture}
\raisebox{1cm}{$= -(\lambda_i^2+\lambda_i^{-2})$}
\begin{tikzpicture}
\draw[fill=gray,opacity=0.2] (0,0) circle (1);
\draw[gray, opacity=0.2, thick] (0,0) circle (1);
\node at (0,0) {\tiny $\bullet$};
\end{tikzpicture}
\caption{Additional relation associated to a puncture.\label{fig:puncturerelation}}
\end{figure}
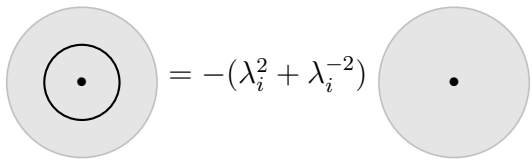

In this paper, we are interested in the relative skein algebra $\Sk_{A,\lambda}(S)$ of a once-punctured torus $S=S_{1,1}$. We will view~$S$ as the surface obtained by identifying opposite sides of a quadrilateral and deleting the image of a vertex of the quadrilateral in the resulting quotient space. Figure~\ref{fig:skeingenerators} shows three simple closed curves on this surface, denoted $\alpha$, $\beta$, and $\gamma$, which will play an important role in our considerations.

\begin{figure}[ht]
\begin{tikzpicture}
\clip(-1.3,-1) rectangle (1.3,1);
\coordinate (a) at (-0.75,-0.75);
\coordinate (b) at (-0.75,0.75);
\coordinate (c) at (0.75,0.75);
\coordinate (d) at (0.75,-0.75);
\draw[gray, thin] (a) -- (b);
\draw[gray, thin] (b) -- (c);
\draw[gray, thin] (c) -- (d);
\draw[gray, thin] (d) -- (a);
\draw[black, thick] (-0.75,0) -- (0.75,0);
\node at (0,0.2) {$\alpha$};
\end{tikzpicture}
\begin{tikzpicture}
\clip(-1.3,-1) rectangle (1.3,1);
\coordinate (a) at (-0.75,-0.75);
\coordinate (b) at (-0.75,0.75);
\coordinate (c) at (0.75,0.75);
\coordinate (d) at (0.75,-0.75);
\draw[gray, thin] (a) -- (b);
\draw[gray, thin] (b) -- (c);
\draw[gray, thin] (c) -- (d);
\draw[gray, thin] (d) -- (a);
\draw[black, thick] (0,-0.75) -- (0,0.75);
\node at (0.2,0) {$\beta$};
\end{tikzpicture}
\begin{tikzpicture}
\clip(-1.3,-1) rectangle (1.3,1);
\coordinate (a) at (-0.75,-0.75);
\coordinate (b) at (-0.75,0.75);
\coordinate (c) at (0.75,0.75);
\coordinate (d) at (0.75,-0.75);
\draw[gray, thin] (a) -- (b);
\draw[gray, thin] (b) -- (c);
\draw[gray, thin] (c) -- (d);
\draw[gray, thin] (d) -- (a);
\draw[black, thick] plot [smooth, tension=1] coordinates { (0,0.75) (0.25,0.25) (0.75,0)};
\draw[black, thick] plot [smooth, tension=1] coordinates { (-0.75,0) (-0.25,-0.25) (0,-0.75)};
\node at (-0.1,-0.1) {$\gamma$};
\end{tikzpicture}
\caption{Skein algebra generators.\label{fig:skeingenerators}}
\end{figure}
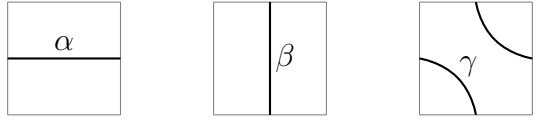

\begin{proposition}
\label{prop:presentation}
The relative skein algebra $\Sk_{A,\lambda}(S_{1,1})$ of the once-punctured torus is the $\mathbb{C}[A^{\pm1},\lambda^{\pm1}]$-algebra generated by $\alpha$, $\beta$, and $\gamma$, subject to the relations 
\begin{align*}
A^{-1}\alpha\beta-A\beta\alpha &= (A^{-2}-A^2)\gamma, \\
A^{-1}\beta\gamma-A\gamma\beta &= (A^{-2}-A^2)\alpha, \\
A^{-1}\gamma\alpha-A\alpha\gamma &= (A^{-2}-A^2)\beta,
\end{align*}
and the relation 
\[
A^{-2}\alpha^2+A^2\beta^2+A^{-2}\gamma^2-A^{-1}\alpha\beta\gamma=A^2+A^{-2}+\lambda+\lambda^{-1}.
\]
\end{proposition}

\begin{proof}
This is immediate from the definition of the relative skein algebra and Proposition~2.3 of~\cite{AS24}.
\end{proof}

\subsection{DAHA and its polynomial representation}

In the following, we will study the skein algebra of the once-punctured torus by relating it to the spherical DAHA of type~$A_1$. Let us set $\mathbb{C}_{q,t}\coloneqq\mathbb{C}(q^{\frac{1}{2}},t^{\frac{1}{2}})$. Then the \emph{double affine Hecke algebra (DAHA)} of type~$A_1$ is the $\mathbb{C}_{q,t}$-algebra $\mathcal{H}_{q,t}(A_1)$ generated by variables $T$, $X^{\pm1}$, and $Y^{\pm1}$, subject to the relations 
\[
TXT=X^{-1}, \quad TY^{-1}T=Y, \quad Y^{-1}X^{-1}YXT^2=q^{-1}, \quad (T-t^{\frac{1}{2}})(T+t^{-\frac{1}{2}})=0.
\]
It follows from the last of these relations that the element $e=(t^{\frac{1}{2}}+t^{-\frac{1}{2}})^{-1}(T+t^{-\frac{1}{2}})$ is an idempotent in~$\mathcal{H}_{q,t}(A_1)$, and we define the \emph{spherical DAHA} of type~$A_1$ to be the algebra $\mathcal{SH}_{q,t}(A_1)=e\mathcal{H}_{q,t}(A_1)e\subset\mathcal{H}_{q,t}(A_1)$. We define elements 
\[
x=(X+X^{-1})e, \quad y=(Y+Y^{-1})e, \quad z=(q^{\frac{1}{2}}YX+q^{-\frac{1}{2}}X^{-1}Y^{-1})e
\]
of $\mathcal{SH}_{q,t}(A_1)$.

\begin{proposition}[\cite{AS24}, Proposition~2.11]
\label{prop:skeinDAHA}
There is a $\mathbb{C}$-algebra embedding $\Sk_{A,\lambda}(S_{1,1})\rightarrow\mathcal{SH}_{q,t}(A_1)$ mapping 
\[
A\mapsto q^{-\frac{1}{2}}, \quad \lambda\mapsto q^{-1}t, \quad \alpha\mapsto x, \quad \beta\mapsto y, \quad \gamma\mapsto z.
\]
\end{proposition}

An important fact is that the DAHA has a faithful representation $\mathcal{H}_{q,t}(A_1)\rightarrow\End\mathbb{C}_{q,t}[X^{\pm1}]$ by $\mathbb{C}_{q,t}$-linear endomorphisms of the Laurent polynomial ring $\mathbb{C}_{q,t}[X^{\pm1}]$. We can restrict this representation to the spherical subalgebra to get a faithful representation $\mathcal{SH}_{q,t}(A_1)\rightarrow\End\mathbb{C}_{q,t}[X^{\pm1}]^{\mathbb{Z}_2}$ where the nontrivial generator of $\mathbb{Z}_2$ acts on the ring $\mathbb{C}_{q,t}[X^{\pm1}]$ by mapping $X\mapsto X^{-1}$. We can compose this representation with the embedding of Proposition~\ref{prop:skeinDAHA} and thereby obtain a $\mathbb{C}$-algebra embedding $\Phi:\Sk_{A,\lambda}(S_{1,1})\rightarrow\End\mathbb{C}_{q,t}[X^{\pm1}]^{\mathbb{Z}_2}$.

\begin{proposition}[\cite{AS24}, Proposition~2.12]
\label{prop:generatingoperators}
The embedding $\Phi:\Sk_{A,\lambda}(S_{1,1})\rightarrow\End\mathbb{C}_{q,t}[X^{\pm1}]^{\mathbb{Z}_2}$ maps $A\mapsto q^{-\frac{1}{2}}$, $\lambda\mapsto q^{-1}t$, and 
\begin{align*}
\alpha &\mapsto X+X^{-1}, \\
\beta &\mapsto V(X)\varpi+V(X^{-1})\varpi^{-1}, \\
\gamma &\mapsto q^{-\frac{1}{2}}X^{-1}V(X)\varpi+q^{-\frac{1}{2}}XV(X^{-1})\varpi^{-1},
\end{align*}
where 
\[
V(X)=\frac{t^{\frac{1}{2}}X-t^{-\frac{1}{2}}X^{-1}}{X-X^{-1}}
\]
and $\varpi$ is the operator defined by $(\varpi f)(X)=f(qX)$ for $f\in\mathbb{C}_{q,t}[X^{\pm1}]^{\mathbb{Z}_2}$.
\end{proposition}

In addition to the curves illustrated in Figure~\ref{fig:skeingenerators}, we will be interested in the family of curves $\gamma_k$~($k\in\mathbb{Z}$) illustrated in Figure~\ref{fig:family}. Note that we have $\gamma_0=\beta$ and $\gamma_1=\gamma$. Under the polynomial representation, these curves map to a family of operators that we can describe explicitly.

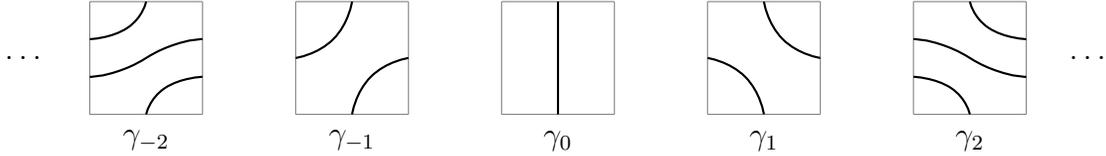
\begin{figure}[ht]
\begin{tikzpicture}
\clip(-1.85,-1.4) rectangle (1.3,1.3);
\coordinate (a) at (-0.75,-0.75);
\coordinate (b) at (-0.75,0.75);
\coordinate (c) at (0.75,0.75);
\coordinate (d) at (0.75,-0.75);
\draw[gray, thin] (a) -- (b);
\draw[gray, thin] (b) -- (c);
\draw[gray, thin] (c) -- (d);
\draw[gray, thin] (d) -- (a);
\draw[black, thick] plot [smooth, tension=1] coordinates { (0,0.75) (-0.25,0.4) (-0.75,0.25)};
\draw[black, thick] plot [smooth, tension=1] coordinates { (0.75,0.25) (0.3,0.15) (-0.3,-0.15) (-0.75,-0.25)};
\draw[black, thick] plot [smooth, tension=1] coordinates { (0.75,-0.25) (0.25,-0.4) (0,-0.75)};
\node at (0,-1.1) {$\gamma_{-2}$};
\node at (-1.6,0) {$\dots$};
\end{tikzpicture}
\begin{tikzpicture}
\clip(-1.3,-1.4) rectangle (1.3,1.3);
\coordinate (a) at (-0.75,-0.75);
\coordinate (b) at (-0.75,0.75);
\coordinate (c) at (0.75,0.75);
\coordinate (d) at (0.75,-0.75);
\draw[gray, thin] (a) -- (b);
\draw[gray, thin] (b) -- (c);
\draw[gray, thin] (c) -- (d);
\draw[gray, thin] (d) -- (a);
\draw[black, thick] plot [smooth, tension=1] coordinates { (0,0.75) (-0.25,0.25) (-0.75,0)};
\draw[black, thick] plot [smooth, tension=1] coordinates { (0.75,0) (0.25,-0.25) (0,-0.75)};
\node at (0,-1.1) {$\gamma_{-1}$};
\end{tikzpicture}
\begin{tikzpicture}
\clip(-1.3,-1.4) rectangle (1.3,1.3);
\coordinate (a) at (-0.75,-0.75);
\coordinate (b) at (-0.75,0.75);
\coordinate (c) at (0.75,0.75);
\coordinate (d) at (0.75,-0.75);
\draw[gray, thin] (a) -- (b);
\draw[gray, thin] (b) -- (c);
\draw[gray, thin] (c) -- (d);
\draw[gray, thin] (d) -- (a);
\draw[black, thick] (0,-0.75) -- (0,0.75);
\node at (0,-1.1) {$\gamma_0$};
\end{tikzpicture}
\begin{tikzpicture}
\clip(-1.3,-1.4) rectangle (1.3,1.3);
\coordinate (a) at (-0.75,-0.75);
\coordinate (b) at (-0.75,0.75);
\coordinate (c) at (0.75,0.75);
\coordinate (d) at (0.75,-0.75);
\draw[gray, thin] (a) -- (b);
\draw[gray, thin] (b) -- (c);
\draw[gray, thin] (c) -- (d);
\draw[gray, thin] (d) -- (a);
\draw[black, thick] plot [smooth, tension=1] coordinates { (0,0.75) (0.25,0.25) (0.75,0)};
\draw[black, thick] plot [smooth, tension=1] coordinates { (-0.75,0) (-0.25,-0.25) (0,-0.75)};
\node at (0,-1.1) {$\gamma_1$};
\end{tikzpicture}
\begin{tikzpicture}
\clip(-1.3,-1.4) rectangle (1.85,1.3);
\coordinate (a) at (-0.75,-0.75);
\coordinate (b) at (-0.75,0.75);
\coordinate (c) at (0.75,0.75);
\coordinate (d) at (0.75,-0.75);
\draw[gray, thin] (a) -- (b);
\draw[gray, thin] (b) -- (c);
\draw[gray, thin] (c) -- (d);
\draw[gray, thin] (d) -- (a);
\draw[black, thick] plot [smooth, tension=1] coordinates { (0,0.75) (0.25,0.4) (0.75,0.25)};
\draw[black, thick] plot [smooth, tension=1] coordinates { (-0.75,0.25) (-0.3,0.15) (0.3,-0.15) (0.75,-0.25)};
\draw[black, thick] plot [smooth, tension=1] coordinates { (-0.75,-0.25) (-0.25,-0.4) (0,-0.75)};
\node at (0,-1.1) {$\gamma_2$};
\node at (1.6,0) {$\dots$};
\end{tikzpicture}
\caption{A family of skein algebra elements.\label{fig:family}}
\end{figure}

\begin{lemma}[\cite{AS24}, Lemma~2.13]
\label{lem:curveoperators}
The embedding $\Phi:\Sk_{A,\lambda}(S_{1,1})\rightarrow\End\mathbb{C}_{q,t}[X^{\pm1}]^{\mathbb{Z}_2}$ maps 
\[
\gamma_k\mapsto q^{-\frac{k}{2}}X^{-k}V(X)\varpi+q^{-\frac{k}{2}}X^kV(X^{-1})\varpi^{-1}
\]
where $V$ and~$\varpi$ are defined as in Proposition~\ref{prop:generatingoperators}.
\end{lemma}

\subsection{$\mathrm{SL}_2(\mathbb{Z})$ symmetries}
\label{sec:SL2Zsymmetries}

Finally, we will study the symmetries of the skein algebra. In general, the mapping class group $\Mod(S)$ of a surface~$S$ acts on the skein algebra~$\Sk_A(S)$. In the case of a once-punctured torus, the mapping class group $\Mod(S_{1,1})\cong\mathrm{SL}_2(\mathbb{Z})$ is generated by transformations $\tau_+$, $\tau_-$, and~$\sigma$, where $\tau_+$ is a Dehn~twist around the loop~$\alpha$, $\tau_-$ is a Dehn~twist around~$\gamma_0=\beta$, and $\sigma=\tau_+\tau_-^{-1}\tau_+$. These generators act on $\Sk_{A,\lambda}(S_{1,1})$ by 
\begin{alignat*}{4}
\tau_+&: \quad &&\alpha\mapsto\alpha, \quad &&\beta\mapsto\gamma_{-1}, \quad &&\gamma\mapsto\beta, \\
\tau_-&: \quad &&\alpha\mapsto\gamma_{-1}, \quad &&\beta\mapsto\beta, \quad &&\gamma\mapsto\alpha, \\
\sigma&: \quad &&\alpha\mapsto\beta, \quad &&\beta\mapsto\alpha, \quad &&\gamma\mapsto\gamma_{-1}.
\end{alignat*}

As explained in Section~0.4.3 of~\cite{C05}, there is also an $\mathrm{SL}_2(\mathbb{Z})$-action on~$\mathcal{H}_{q,t}(A_1)$, which is given on generators by 
\begin{alignat*}{4}
\tau_+&: \quad &&X\mapsto X, \quad &&Y\mapsto q^{-\frac{1}{2}}XY, \quad &&T\mapsto T, \\
\tau_-&: \quad &&X\mapsto q^{\frac{1}{2}}YX, \quad &&Y\mapsto Y, \quad &&T\mapsto T, \\
\sigma&: \quad &&X\mapsto Y^{-1}, \quad &&Y\mapsto XT^2, \quad &&T\mapsto T.
\end{alignat*}
This action preserves the spherical subalgebra $\mathcal{SH}_{q,t}(A_1)\subset\mathcal{H}_{q,t}(A_1)$. By equation~2.53 of~\cite{GKNPS23}, the generators of $\mathrm{SL}_2(\mathbb{Z})$ act by 
\begin{alignat*}{4}
\tau_+&: \quad &&x\mapsto x, \quad &&y\mapsto z', \quad &&z\mapsto y, \\
\tau_-&: \quad &&x\mapsto z', \quad &&y\mapsto y, \quad &&z\mapsto x, \\
\sigma&: \quad &&x\mapsto y, \quad &&y\mapsto x, \quad &&z\mapsto z',
\end{alignat*}
where $z'=\frac{xy+yx}{q^{\frac{1}{2}}+q^{-\frac{1}{2}}}-z$.

\begin{proposition}
\label{prop:SL2Zequivariance}
The embedding $\Sk_{A,\lambda}(S_{1,1})\rightarrow\mathcal{SH}_{q,t}(A_1)$ of Proposition~\ref{prop:skeinDAHA} is $\mathrm{SL}_2(\mathbb{Z})$-equivariant.
\end{proposition}

\begin{proof}
The skein relation in Figure~\ref{subfig:resolution} implies that $\gamma_{-1}=\frac{\alpha\beta+\beta\alpha}{A+A^{-1}}-\gamma$. Hence, by Proposition~\ref{prop:skeinDAHA}, the embedding maps $\alpha\mapsto x$, $\beta\mapsto y$, $\gamma\mapsto z$, $\gamma_{-1}\mapsto z'$. The desired statement then follows from the formulas given above for the $\mathrm{SL}_2(\mathbb{Z})$-actions on~$\Sk_{A,\lambda}(S_{1,1})$ and~$\mathcal{SH}_{q,t}(A_1)$.
\end{proof}

Finally, note that we have three involutions $\xi_1$,~$\xi_2$,~$\xi_3:\Sk_{A,\lambda}(S_{1,1})\rightarrow\Sk_{A,\lambda}(S_{1,1})$ defined on the generators $\alpha$, $\beta$, $\gamma$ as in Section~\ref{sec:StatementOfTheMainResult}.

\begin{proposition}
\label{prop:skeindualities}
$\tau_+$, $\tau_-$, and $\sigma$ provide isomorphisms between the $\xi_i$-invariant subalgebras of~$\Sk_{A,\lambda}(S_{1,1})$ as in Figure~\ref{fig:skeindualities}.
\end{proposition}

\begin{proof}
If $p\in\Sk_{A,\lambda}(S_{1,1})^{\xi_1}$ then, by Proposition~\ref{prop:presentation}, $p$ can be written as a $\xi_1$-invariant polynomial in $\alpha$, $\beta$, and $\gamma$ with coefficients in~$\mathbb{C}[A^{\pm1},\lambda^{\pm1}]$. We have $\tau_+(\alpha)=\alpha$, $\tau_+(\beta)=\gamma_{-1}$, $\tau_+(\gamma)=\beta$, and we know that $\gamma_{-1}=\frac{\alpha\beta+\beta\alpha}{A+A^{-1}}-\gamma$ by the skein relation in Figure~\ref{subfig:resolution}. Hence $\tau_+$ takes $p$ to a $\xi_3$-invariant polynomial $\tau_+(p)\in\Sk_{A,\lambda}(S_{1,1})^{\xi_3}$. In a similar way, one can check that all the arrows in Figure~\ref{fig:skeindualities} map between invariant subalgebras.
\end{proof}

\section{Quantized $K$-theoretic Coulomb branches}

In this section, we review the notion of a quantized $K$-theoretic Coulomb branch from the work of Braverman, Finkelberg, and Nakajima~\cite{BFN18}.

\subsection{The variety of triples}

In the following, we adopt the standard notation $\mathcal{K}=\mathbb{C}(\!( z)\!)$ and $\mathcal{O}=\mathbb{C}\llbracket z\rrbracket$. Given a complex, connected, reductive algebraic group~$G$, we write $G_{\mathcal{K}}$ and $G_\mathcal{O}$ for the $\mathcal{K}$- and $\mathcal{O}$-valued points of~$G$, respectively. Then the \emph{affine Grassmannian} of~$G$ is defined as 
\[
\Gr_G\coloneqq G_{\mathcal{K}}/G_{\mathcal{O}}.
\]
This object has the structure of an ind-scheme representing an explicit functor of points (see Section~2 of~\cite{BFN18}). In the following, we equip $\Gr_G$ with the reduced ind-scheme structure.

In~\cite{BFN18} Braverman, Finkelberg, and Nakajima consider the following natural generalization of the affine Grassmannian. Given a complex representation~$N$ of~$G$, we write $N_{\mathcal{K}}=N\otimes_{\mathbb{C}}\mathcal{K}$ and $N_{\mathcal{O}}=N\otimes_{\mathbb{C}}\mathcal{O}$. We consider the space 
\[
\mathcal{T}_{G,N}\coloneqq G_{\mathcal{K}}\times_{G_{\mathcal{O}}}N_{\mathcal{O}}
\]
where the right hand side is the quotient of $G_{\mathcal{K}}\times N_{\mathcal{O}}$ by the equivalence relation~$\sim$ defined by $(g,n)\sim(gb^{-1},bn)$ for $b\in G_{\mathcal{O}}$. Let us write $[g,n]$ for the class of a pair $(g,n)\in G_{\mathcal{K}}\times N_{\mathcal{O}}$ in this quotient space. Given such a pair one has in general $gn\in N_{\mathcal{K}}$, and the \emph{variety of triples} is the subspace 
\[
\mathcal{R}_{G,N}\coloneqq\{[g,n]\in\mathcal{T}_{G,N}:gn\in N_{\mathcal{O}}\}
\]
of~$\mathcal{T}_{G,N}$. Like the affine Grassmannian, the spaces $\mathcal{T}_{G,N}$ and $\mathcal{R}_{G,N}$ are ind-schemes, and in the following we take the reduced ind-scheme structure. Using the functor of points description in Section~2 of~\cite{BFN18}, one can view $\mathcal{R}_{G,N}$ as a moduli space of triples of gauge theoretic data on a formal disk, justifying the terminology. When there is no possibility of confusion, we will sometimes omit the subscripts and write $\mathcal{T}=\mathcal{T}_{G,N}$ and $\mathcal{R}=\mathcal{R}_{G,N}$.

Note that there is an action of $G_{\mathcal{K}}$ on the space $\mathcal{T}_{G,N}$ by left multiplication, and that the subgroup~$G_{\mathcal{O}}$ preserves $\mathcal{R}_{G,N}$. Below we will construct the equivariant $K$-theory of $\mathcal{R}_{G,N}$ for an extension of this action. We will then define the quantized Coulomb branch in terms of this $K$-theory, following~\cite{BFN18}.

\subsection{Equivariant $K$-theory}

To define the equivariant $K$-theory of the variety of triples, we use the fact that there is a natural $G_\mathcal{O}$-action on the affine Grassmannian $\Gr_G$ and a decomposition 
\[
\Gr_G=\bigsqcup_{\lambda\in X_+}\Gr_G^\lambda
\]
of this space into $G_{\mathcal{O}}$-orbits $\Gr_G^\lambda$ parametrized by dominant coweights $\lambda$ of~$G$. The closure $\overline{\Gr}_G^\lambda$ of~$\Gr_G^\lambda$ is a projective variety. Moreover, using the standard partial order on the set~$X_+$ of dominant coweights, one has a decomposition $\overline{\Gr}_G^\lambda=\bigsqcup_{\mu\leq\lambda}\Gr_G^\mu$. We consider the spaces $\mathcal{T}_{\leq\lambda}=\pi^{-1}\left(\overline{\Gr}_G^\lambda\right)$, where $\pi:\mathcal{T}\rightarrow\Gr_G$ is the projection, and we define $\mathcal{R}_{\leq\lambda}=\mathcal{R}\cap\mathcal{T}_{\leq\lambda}$.

We can further decompose the spaces $\mathcal{T}_{\leq\lambda}$ and $\mathcal{R}_{\leq\lambda}$ as follows. For any integer $d>0$, we form the quotient $\mathcal{T}^d=G_{\mathcal{K}}\times_{G_{\mathcal{O}}}(N_{\mathcal{O}}/z^dN_{\mathcal{O}})$ of~$\mathcal{T}$ by fiberwise translation by $z^dN_{\mathcal{O}}$. Then for~$e>d$ there is an induced map $p_d^e:\mathcal{T}^e\rightarrow\mathcal{T}^d$, and we recover the space~$\mathcal{T}$ as the inverse limit of this system. We will write $\mathcal{T}_{\leq\lambda}^d$ for the image of~$\mathcal{T}_{\leq\lambda}$ in the quotient~$\mathcal{T}^d$. Then the $p_d^e$ restrict to maps $p_d^e:\mathcal{T}_{\leq\lambda}^e\rightarrow\mathcal{T}_{\leq\lambda}^d$, and we recover $\mathcal{T}_{\leq\lambda}$ as the inverse limit of this system.

One can show that for $d\in\mathbb{Z}$ greater than a constant depending on~$\lambda$, the action of $z^dN_{\mathcal{O}}$ on~$\mathcal{T}$ stabilizes~$\mathcal{R}_{\leq\lambda}$. Hence for $d\gg0$ we can take the quotient $\mathcal{R}_{\leq\lambda}^d$ of $\mathcal{R}_{\leq\lambda}$ by this action. This quotient is a closed subset of $\mathcal{T}_{\leq\lambda}^d$. For $e>d$ the $p_d^e$ restrict to maps $p_d^e:\mathcal{R}_{\leq\lambda}^e\rightarrow\mathcal{R}_{\leq\lambda}^d$, and we recover $\mathcal{R}_{\leq\lambda}$ as the inverse limit of this system.

Now suppose that the action of~$G$ on~$N$ extends to an action of $G\times\mathbb{C}^*$ on~$N$ where the $\mathbb{C}^*$-factor acts by scalar multiplication. In physics, the group $G$ is called a \emph{gauge group} while $\mathbb{C}^*$ is called a \emph{flavor symmetry group}. Then we get a $(G\times\mathbb{C}^*)_{\mathcal{O}}\rtimes\mathbb{C}^*$-action on~$\mathcal{R}$ where the second $\mathbb{C}^*$-factor acts by rescaling the variable~$z$. This gives a $(G\times\mathbb{C}^*)_{\mathcal{O}}\rtimes\mathbb{C}^*$-action on $\mathcal{R}_{\leq\lambda}^d$. It factors through an action of the quotient $(G_i\times\mathbb{C}^*)\rtimes\mathbb{C}^*$ for sufficiently large~$i$ where $G_i=G(\mathcal{O}/z^i\mathcal{O})$. The space $\mathcal{R}_{\leq\lambda}^d$ is an ordinary variety, so for an algebraic group $\Gamma$, we can talk about its $\Gamma$-equivariant $K$-theory as defined for example in~\cite{CG97}.

\begin{definition}[\cite{BFN18,VV10}]
\label{def:KtheoryR}
The $(G\times\mathbb{C}^*)_{\mathcal{O}}\rtimes\mathbb{C}^*$-equivariant $K$-theory of $\mathcal{R}_{G,N}$ is the limit 
\[
K^{(G\times\mathbb{C}^*)_{\mathcal{O}}\rtimes\mathbb{C}^*}\left(\mathcal{R}_{G,N}\right)\coloneqq\lim_{\lambda}K^{(G\times\mathbb{C}^*)_{\mathcal{O}}\rtimes\mathbb{C}^*}\left(\mathcal{R}_{\leq\lambda}\right)
\]
where $K^{(G\times\mathbb{C}^*)_{\mathcal{O}}\rtimes\mathbb{C}^*}\left(\mathcal{R}_{\leq\lambda}\right)\coloneqq K^{(G_i\times\mathbb{C}^*)\rtimes\mathbb{C}^*}(\mathcal{R}_{\leq\lambda}^d)$ for some $d$,~$i\gg0$ and the limit diagram consists of all pushforward maps $K^{(G\times\mathbb{C}^*)_{\mathcal{O}}\rtimes\mathbb{C}^*}\left(\mathcal{R}_{\leq\lambda}\right)\rightarrow K^{(G\times\mathbb{C}^*)_{\mathcal{O}}\rtimes\mathbb{C}^*}\left(\mathcal{R}_{\leq\mu}\right)$ induced by embeddings $\mathcal{R}_{\leq\lambda}\hookrightarrow\mathcal{R}_{\leq\mu}$ for $\lambda\leq\mu$.
\end{definition}

Following the approach of Remark 3.9(3) in~\cite{BFN18}, we can define a convolution product~$*$ on the vector space $K^{(G\times\mathbb{C}^*)_{\mathcal{O}}\rtimes\mathbb{C}^*}\left(\mathcal{R}_{G,N}\right)$ to get a noncommutative algebra over $K^{\mathbb{C}^*}(\mathrm{pt})\cong\mathbb{C}[q^{\pm1}]$. In the following, we will actually replace the second $\mathbb{C}^*$-factor by its standard double cover and thereby extend scalars to $\mathbb{C}[q^{\pm\frac{1}{2}}]$.

\begin{definition}[\cite{BFN18}]
The algebra with underlying vector space $K^{(G\times\mathbb{C}^*)_{\mathcal{O}}\rtimes\mathbb{C}^*}\left(\mathcal{R}_{G,N}\right)$ and multiplication~$*$ is the \emph{quantized $K$-theoretic Coulomb branch} with $\mathbb{C}^*$ flavor symmetry associated to the gauge group~$G$ and representation~$N$.
\end{definition}

\subsection{The abelian case}
\label{sec:TheAbelianCase}

Let us consider the group $\mathrm{GL}_2$ and its diagonal subgroup $T\subset\mathrm{GL}_2$. There are bijections $\mathcal{R}_{T,0}\cong\Gr_T\cong T_{\mathcal{K}}/T_{\mathcal{O}}$. Note that $\mathcal{K}^*=\mathcal{K}\setminus\{0\}$ while $\mathcal{O}^*$ is the  set of formal power series with nonzero constant term. If~$f=\sum_{i\geq m}a_iz^i$ is a formal Laurent series in $\mathcal{K}^*$ with $a_m\neq0$, then we have $f=z^mg$ where $g=\sum_{i\geq0}a_{m+i}z^i$ is an element of~$\mathcal{O}^*$. It follows that any element of~$\mathcal{R}_{T,0}$ can be represented by a matrix of the form 
\begin{equation}
\label{eqn:matrix}
\begin{pmatrix}
z^{m_1} & 0 \\
0 & z^{m_2}
\end{pmatrix}
\in T_{\mathcal{K}}
\end{equation}
for some $m_1$,~$m_2\in\mathbb{Z}$. Thus the space $\mathcal{R}_{T,0}$ is naturally identified with the lattice $X_*(T)$ of cocharacters of~$T$ with the discrete topology.

We would like to compute the $(T\times\mathbb{C}^*)_{\mathcal{O}}\rtimes\mathbb{C}^*$-equivariant $K$-theory of this space. Since $T_{\mathcal{O}}$ acts on $\mathcal{R}_{T,0}=T_{\mathcal{K}}/T_{\mathcal{O}}$ by left multiplication and $T_{\mathcal{K}}$ is abelian, we know that $(T\times\mathbb{C}^*)_{\mathcal{O}}$ acts trivially on~$\mathcal{R}_{T,0}$. The action of~$\mathbb{C}^*$ on~$\mathcal{R}_{T,0}$ by loop rotation is also trivial. Therefore $(T\times\mathbb{C}^*)_i\rtimes\mathbb{C}^*$ acts trivially on $\mathcal{R}_{\leq\lambda}^d$ for any $i\geq1$, and by definition we have $K^{(T\times\mathbb{C}^*)_{\mathcal{O}}\times\mathbb{C}^*}(\mathcal{R}_{\leq\lambda})=K^{T\times\mathbb{C}^*\rtimes\mathbb{C}^*}(\mathcal{R}_{\leq\lambda}^d)$. Then we have 
\[
K^{(T\times\mathbb{C}^*)_{\mathcal{O}}\rtimes\mathbb{C}^*}(\mathcal{R}_{T,0})\cong\bigoplus_{\mu\in X_*(T)}K^{T\times\mathbb{C}^*\times\mathbb{C}^*}(\mathrm{pt})
\]
as vector spaces where $\mathrm{pt}$ is the variety consisting of a single point. Note that because the actions of $T\times\mathbb{C}^*$ and the second~$\mathbb{C}^*$ commute, the semidirect product is actually direct.

Now for any linear algebraic group~$\Gamma$, the $\Gamma$-equivariant $K$-theory of a point is the group algebra $\mathbb{C}[X^*(\Gamma)]$ of the character lattice $X^*(\Gamma)$ (see~\cite{CG97}, Section~5.2.1). We will write 
\[
\mathbb{C}[X^*(T\times\mathbb{C}^*\times\mathbb{C}^*)]\cong\mathbb{C}[q^{\pm1},z^{\pm1},w_1^{\pm1},w_2^{\pm1}].
\]
Here $w_1$ and~$w_2$ correspond to the characters of~$T$ given by $\diag(t_1,t_2)\mapsto t_1$ and $\diag(t_1,t_2)\mapsto t_2$, respectively, and $z$ and~$q$ correspond to the characters projecting onto the first and second $\mathbb{C}^*$-factors, respectively. We will also write $\mathbb{C}[X_*(T)]\cong\mathbb{C}[D_1^{\pm1},D_2^{\pm1}]$ where $D_1$ and~$D_2$ correspond to the cocharacters $t\mapsto\diag(t,1)$ and $t\mapsto\diag(1,t)$, respectively. In this way, we see that $K^{(T\times\mathbb{C}^*)_{\mathcal{O}}\rtimes\mathbb{C}^*}(\mathcal{R}_{T,0})$ is generated as a $\mathbb{C}$-algebra by variables $q^{\pm1}$, $z^{\pm1}$, $w_1^{\pm1}$, $w_2^{\pm1}$, $D_1^{\pm1}$, and~$D_2^{\pm1}$. In~fact, if we write $\mathcal{D}_{q,z}^0$ for the $\mathbb{C}[q^{\pm1},z^{\pm1}]$-algebra generated by variables $w_1^{\pm1}$, $w_2^{\pm1}$, $D_1^{\pm1}$,~$D_2^{\pm1}$, subject to the commutation relations 
\[
[D_1,D_2]=[w_1,w_2]=0, \quad D_rw_s=q^{2\delta_{rs}}w_sD_r,
\]
then there is a $\mathbb{C}$-algebra isomorphism $K^{(T\times\mathbb{C}^*)_{\mathcal{O}}\rtimes\mathbb{C}^*}(\mathcal{R}_{T,0})\cong\mathcal{D}_{q,z}^0$.

Suppose now that $N$ is a representation of~$\mathrm{GL}_2$, and let us use the same symbol for the restriction of this representation to the subgroup~$T$. Then by restricting the embedding $\mathcal{T}_{T,N}\cong T_{\mathcal{K}}\times_{T_{\mathcal{O}}}N_{\mathcal{O}}\hookrightarrow\mathrm{GL}_2(\mathcal{K})\times_{\mathrm{GL}_2(\mathcal{O})}N_{\mathcal{O}}\cong\mathcal{T}_{\mathrm{GL}_2,N}$, we obtain a map 
\[
\iota:\mathcal{R}_{T,N}\rightarrow\mathcal{R}_{\mathrm{GL}_2,N}.
\]
On the other hand, the zero section of the projection $\pi:\mathcal{T}_{T,N}\rightarrow\Gr_T$ provides a map 
\[
\mathbf{z}:\Gr_T\rightarrow\mathcal{R}_{T,N}.
\]
Since we have $\mathcal{R}_{T,0}\cong\Gr_T$, we can think of $\mathbf{z}$ as a map between varieties of triples. Let us write $\mathcal{D}_{q,z}$ for the localization of~$\mathcal{D}_{q,z}^0$ at the multiplicative subset generated by the elements~$w_1-q^kw_2$ for~$k\in\mathbb{Z}$ and $1-q^k$ for~$k\in\mathbb{Z}\setminus\{0\}$. As explained in Section~3.3 of our previous paper~\cite{AS24}, these maps~$\iota$ and~$\mathbf{z}$ provide a $\mathbb{C}$-algebra embedding 
\begin{equation}
\label{eqn:embedding}
\mathbf{z}^*(\iota_*)^{-1}:K^{(\mathrm{GL}_2\times\mathbb{C}^*)_{\mathcal{O}}\rtimes\mathbb{C}^*}(\mathcal{R}_{\mathrm{GL}_2,N})\hookrightarrow\mathcal{D}_{q,z}.
\end{equation}
In the following, this embedding will be our main tool for studying the quantized Coulomb branch $K^{(\mathrm{GL}_2\times\mathbb{C}^*)_{\mathcal{O}}\rtimes\mathbb{C}^*}(\mathcal{R}_{\mathrm{GL}_2,N})$. We note that the analogous map on equivariant Borel--Moore homology was discussed in Appendix~A of~\cite{BFN19}.

\section{Calculation of Coulomb branches}

In this section, we study Coulomb branches associated to the group $\mathrm{PGL}_2$ and explicitly describe the Coulomb branch appearing in our main result.

\subsection{The variety of triples for $\mathrm{PGL}_2$}

We now assume that $N$ is a representation of the group~$\mathrm{PGL}_2$. We use the same symbol to denote the representation of~$\mathrm{GL}_2$ obtained by composing with the projection $p:\mathrm{GL}_2\rightarrow\mathrm{PGL}_2$. Below we wish to relate the varieties $\mathcal{R}_{\mathrm{GL}_2,N}$ and~$\mathcal{R}_{\mathrm{PGL}_2,N}$. We first note that there is a $\mathbb{Z}$-action on $\mathcal{T}_{\mathrm{GL}_2,N}$ where an integer $m\in\mathbb{Z}$ acts by 
\[
m\cdot[g,n]=[z^mg,z^{-m}n]
\]
for $g\in\mathrm{GL}_2(\mathcal{K})$ and $n\in N_{\mathcal{O}}$. This action clearly preserves the subspace $\mathcal{R}_{\mathrm{GL}_2,N}$. We also have a $\mathbb{Z}$-action on $\Gr_{\mathrm{GL}_2}=\mathrm{GL}_2(\mathcal{K})/\mathrm{GL}_2(\mathcal{O})$ where an integer $m\in\mathbb{Z}$ acts by $m\cdot\left(g\,\mathrm{GL}_2(\mathcal{O})\right)=z^mg\,\mathrm{GL}_2(\mathcal{O})$ for $g\in\mathrm{GL}_2(\mathcal{K})$.

\begin{lemma}
\label{lem:Zaction}
The projection $\pi:\mathcal{R}_{\mathrm{GL}_2,N}\rightarrow\Gr_{\mathrm{GL}_2}$ intertwines these $\mathbb{Z}$-actions. The induced map on quotients is the projection $\pi:\mathcal{R}_{\mathrm{PGL}_2,N}\rightarrow\Gr_{\mathrm{PGL}_2}$.
\end{lemma}

\begin{proof}
The first statement is clear from the definition of the $\mathbb{Z}$-actions. To prove the second statement, we consider the map $\mathcal{T}_{\mathrm{GL}_2,N}\rightarrow\mathcal{T}_{\mathrm{PGL}_2,N}$ given by $[g,n]\mapsto[p(g),n]$. A simple calculation shows that it is well defined. Abusing notation, we will denote it also by~$p$. We then observe that if $p\left([g_1,n_1]\right)=p\left([g_2,n_2]\right)$ then there exists $b\in\mathrm{PGL}_2(\mathcal{O})$ such that $(p(g_1),n_1)=(p(g_2)b^{-1},bn_2)$. Thus there exists $b'\in\mathrm{GL}_2(\mathcal{O})$ and $c\in\mathcal{K}^*$ such that $(g_1,n_1)=(cg_2b'^{-1},c^{-1}b'n_2)$. We can write $c=z^mc'$ for some $c'\in\mathcal{O}^*$ and some integer~$m$. Absorbing $c'$ into~$b'$, we may assume without loss of generality that $(g_1,n_1)=(z^mg_2b'^{-1},z^{-m}b'n_2)$. This means that $[g_1,n_1]$ and~$[g_2,n_2]$ are related by the $\mathbb{Z}$-action, and hence the map $p:\mathcal{T}_{\mathrm{GL}_2,N}\rightarrow\mathcal{T}_{\mathrm{PGL}_2,N}$ descends to an isomorphism 
\[
\mathcal{T}_{\mathrm{GL}_2,N}/\mathbb{Z}\cong\mathcal{T}_{\mathrm{PGL}_2,N}.
\]
One can likewise define a map $\Gr_{\mathrm{GL}_2}\rightarrow\Gr_{\mathrm{PGL}_2}$ given by $g\,\mathrm{GL}_2(\mathcal{O})\mapsto p(g)\,\mathrm{PGL}_2(\mathcal{O})$, and a similar argument shows that it descends to an isomorphism $\Gr_{\mathrm{GL}_2}/\mathbb{Z}\cong\Gr_{\mathrm{PGL}_2}$. Finally, from the previous isomorphism, we obtain 
\begin{align*}
\mathcal{R}_{\mathrm{PGL}_2,N} &= \{[\bar{g},n]\in\mathcal{T}_{\mathrm{PGL}_2,N}:\bar{g}n\in N_{\mathcal{O}}\} \\
&\cong \{[g,n]\in\mathcal{T}_{\mathrm{GL}_2,N}: gn\in N_{\mathcal{O}}\}/\mathbb{Z} \\
&=\mathcal{R}_{\mathrm{GL}_2,N}/\mathbb{Z}.
\end{align*}
This completes the proof.
\end{proof}

There is a map $\nu:\Gr_{\mathrm{GL}_2}\rightarrow\mathbb{Z}$ sending the class of $g\in\mathrm{GL}_2(\mathcal{K})$ to the unique integer $\nu(g)\in\mathbb{Z}$ such that $z^{-\nu(g)}\det(g)\in\mathcal{O}^*$. The connected components of $\Gr_{\mathrm{GL}_2}$ are parametrized by~$\mathbb{Z}$ in such a way that $\nu^{-1}(k)$ is the component corresponding to $k\in\mathbb{Z}$. In particular, the generator $1\in\mathbb{Z}$ acts by sending the $k$th component isomorphically to the $(k+2)$nd component. Let us write $\mathcal{R}_k\subset\mathcal{R}_{\mathrm{GL}_2,N}$ for the preimage of the $k$th component of $\Gr_{\mathrm{GL}_2}$ under the map $\pi:\mathcal{R}_{\mathrm{GL}_2,N}\rightarrow\Gr_{\mathrm{GL}_2}$. Then we have the decompositions 
\[
\mathcal{R}_{\mathrm{GL}_2,N}=\bigsqcup_{k\in\mathbb{Z}}\mathcal{R}_k, \qquad K^{(\mathrm{GL}_2\times\mathbb{C}^*)_{\mathcal{O}}\rtimes\mathbb{C}^*}(\mathcal{R}_{\mathrm{GL}_2,N})\cong\bigoplus_{k\in\mathbb{Z}}K^{(\mathrm{GL}_2\times\mathbb{C}^*)_{\mathcal{O}}\rtimes\mathbb{C}^*}(\mathcal{R}_k).
\]
By the equivariance in Lemma~\ref{lem:Zaction}, the generator $1\in\mathbb{Z}$ gives an isomorphism $\mathcal{R}_k\cong\mathcal{R}_{k+2}$. This induces an isomorphism $K^{(\mathrm{GL}_2\times\mathbb{C}^*)_{\mathcal{O}}\rtimes\mathbb{C}^*}(\mathcal{R}_k)\cong K^{(\mathrm{GL}_2\times\mathbb{C}^*)_{\mathcal{O}}\rtimes\mathbb{C}^*}(\mathcal{R}_{k+2})$, and so we have a $\mathbb{Z}$-action on $K^{(\mathrm{GL}_2\times\mathbb{C}^*)_{\mathcal{O}}\rtimes\mathbb{C}^*}(\mathcal{R}_{\mathrm{GL}_2,N})$. 

\begin{lemma}
\label{lem:PGL2spaceoftriples}
We have $K^{(\mathrm{GL}_2\times\mathbb{C}^*)_{\mathcal{O}}\rtimes\mathbb{C}^*}(\mathcal{R}_{\mathrm{PGL}_2,N})\cong K^{(\mathrm{GL}_2\times\mathbb{C}^*)_{\mathcal{O}}\rtimes\mathbb{C}^*}(\mathcal{R}_{\mathrm{GL}_2,N})_\mathbb{Z}$, and the embedding~\eqref{eqn:embedding} induces an embedding 
\[
K^{(\mathrm{GL}_2\times\mathbb{C}^*)_{\mathcal{O}}\rtimes\mathbb{C}^*}(\mathcal{R}_{\mathrm{PGL}_2,N})\hookrightarrow\mathcal{D}_{q,z}/(D_1D_2-1)
\]
with image $K^{(\mathrm{GL}_2\times\mathbb{C}^*)_{\mathcal{O}}\rtimes\mathbb{C}^*}(\mathcal{R}_{\mathrm{GL}_2,N})/(D_1D_2-1)$.
\end{lemma}

\begin{proof}
Since the $\mathbb{Z}$-action identifies $\mathcal{R}_k\cong\mathcal{R}_{k+2}$, Lemma~\ref{lem:Zaction} implies that $\mathcal{R}_{\mathrm{PGL}_2,N}\cong\mathcal{R}_0\sqcup\mathcal{R}_1$. Therefore 
\begin{align*}
K^{(\mathrm{GL}_2\times\mathbb{C}^*)_{\mathcal{O}}\rtimes\mathbb{C}^*}(\mathcal{R}_{\mathrm{PGL}_2,N}) &\cong K^{(\mathrm{GL}_2\times\mathbb{C}^*)_{\mathcal{O}}\rtimes\mathbb{C}^*}(\mathcal{R}_0\sqcup\mathcal{R}_1) \\
&\cong K^{(\mathrm{GL}_2\times\mathbb{C}^*)_{\mathcal{O}}\rtimes\mathbb{C}^*}(\mathcal{R}_0)\oplus K^{(\mathrm{GL}_2\times\mathbb{C}^*)_{\mathcal{O}}\rtimes\mathbb{C}^*}(\mathcal{R}_1) \\
&\cong K^{(\mathrm{GL}_2\times\mathbb{C}^*)_{\mathcal{O}}\rtimes\mathbb{C}^*}(\mathcal{R}_{\mathrm{GL}_2,N})_\mathbb{Z}.
\end{align*}
In the same way that we defined the $\mathbb{Z}$-action on $\mathcal{R}_{\mathrm{GL}_2,N}$, we can define $\mathbb{Z}$-actions on the spaces $\mathcal{R}_{T,N}$ and $\mathcal{R}_{T,0}$. The generator $1\in\mathbb{Z}$ acts on the latter space by rescaling the matrix~\eqref{eqn:matrix} by~$z$. The induced map on $K^{(T\times\mathbb{C}^*)_{\mathcal{O}}\rtimes\mathbb{C}^*}(\mathcal{R}_{T,0})\cong\mathcal{D}_{q,z}^0$ is given by multiplication by $D_1D_2$. The maps $\iota$ and~$\mathbf{z}$ are $\mathbb{Z}$-equivariant, so 
\begin{align*}
K^{(\mathrm{GL}_2\times\mathbb{C}^*)_{\mathcal{O}}\rtimes\mathbb{C}^*}(\mathcal{R}_{\mathrm{PGL}_2,N}) &\cong K^{(\mathrm{GL}_2\times\mathbb{C}^*)_{\mathcal{O}}\rtimes\mathbb{C}^*}(\mathcal{R}_{\mathrm{GL}_2,N})_\mathbb{Z} \\
&\hookrightarrow\mathcal{D}_{q,z}/(D_1D_2-1),
\end{align*}
and the image of this embedding is $K^{(\mathrm{GL}_2\times\mathbb{C}^*)_{\mathcal{O}}\rtimes\mathbb{C}^*}(\mathcal{R}_{\mathrm{GL}_2,N})/(D_1D_2-1)$.
\end{proof}

\subsection{Calculation of equivariant $K$-theory}

We now want to compute the $(\mathrm{PGL}_2\times\mathbb{C}^*)_{\mathcal{O}}\rtimes\mathbb{C}^*$-equivariant $K$-theory of~$\mathcal{R}_{\mathrm{PGL}_2,N}$. We begin with some general considerations. Let $X$ be a variety and $G$ an algebraic group acting on~$X$. Recall that a $G$-equivariant sheaf on~$X$ is defined to be a sheaf~$\mathcal{F}$ of $\mathcal{O}_X$-modules on~$X$ together with an isomorphism 
\[
I:a^*\mathcal{F}\stackrel{\sim}{\longrightarrow}p^*\mathcal{F}
\]
of sheaves on~$G\times X$ where $a:G\times X\rightarrow X$ is the group action and $p:G\times X\rightarrow X$ is the projection. This map~$I$ is required to satisfy a cocycle condition (see~\cite{CG97}, Section~5.1). At the level of stalks, $I$~induces an isomorphism $\mathcal{F}_{g\cdot x}\cong(a^*\mathcal{F})_{(g,x)}\stackrel{\sim}{\rightarrow}(p^*\mathcal{F})_{(g,x)}\cong\mathcal{F}_x$ for~$g\in G$ and~$x\in X$.

Now suppose that the $G$-action on~$X$ descends to a $G/H$-action for some normal subgroup $H\subset G$. Let~$\mathcal{F}$ be a sheaf of $\mathcal{O}_X$-modules with equivariant structure defined by an isomorphism~$I$ as above. Since the group $H$ acts trivially on~$X$, we have $\mathcal{F}_{h\cdot x}=\mathcal{F}_x$ for any~$h\in H$ and~$x\in X$. Then $(\mathcal{F},I)$ descends to a $G/H$-equivariant sheaf on~$X$ if and only if the isomorphism $\mathcal{F}_x\stackrel{\sim}{\rightarrow}\mathcal{F}_x$ induced by~$I$ on stalks at $(h,x)$ is the identity map for any~$h\in H$ and~$x\in X$. The category of $G/H$-equivariant coherent sheaves is in fact a Serre subcategory of the category of $G$-equivariant coherent sheaves, and thus $K^{G/H}(X)$ is identified with a subspace of~$K^G(X)$.

\begin{lemma}
\label{lem:GL2PGL2}
$K^{(\mathrm{PGL}_2\times\mathbb{C}^*)_{\mathcal{O}}\rtimes\mathbb{C}^*}(\mathcal{R}_{\mathrm{PGL}_2,N})$ is identified with a subspace of $K^{(\mathrm{GL}_2\times\mathbb{C}^*)_{\mathcal{O}}\rtimes\mathbb{C}^*}(\mathcal{R}_{\mathrm{PGL}_2,N})$, and the embedding of Lemma~\ref{lem:PGL2spaceoftriples} restricts to an embedding 
\[
K^{(\mathrm{PGL}_2\times\mathbb{C}^*)_{\mathcal{O}}\rtimes\mathbb{C}^*}(\mathcal{R}_{\mathrm{PGL}_2,N})\hookrightarrow\mathcal{D}_{q,z}^{\mathbb{C}^*}/(D_1D_2-1)
\]
where $\mathbb{C}^*$ acts on~$\mathcal{D}_{q,z}$ by simultaneously rescaling $w_1$ and~$w_2$. Moreover, the image of this embedding is $K^{(\mathrm{GL}_2\times\mathbb{C}^*)_{\mathcal{O}}\rtimes\mathbb{C}^*}(\mathcal{R}_{\mathrm{GL}_2,N})^{\mathbb{C}^*}/(D_1D_2-1)$.
\end{lemma}

\begin{proof}
There is a natural surjective group homomorphism $p:\mathrm{GL}_2(\mathcal{O})\rightarrow\mathrm{PGL}_2(\mathcal{O})$, and for any integer~$i\geq1$, surjective homomorphisms $\mathrm{GL}_2(\mathcal{O})\rightarrow\mathrm{GL}_2(\mathcal{O}/z^i\mathcal{O})$ and $\mathrm{PGL}_2(\mathcal{O})\rightarrow\mathrm{PGL}_2(\mathcal{O}/z^i\mathcal{O})$. Let us write~$H_i$ and~$\bar{H}_i$, respectively, for the kernels of these last two maps. If $h\in H_i$ then we have $p(h)\in\bar{H}_i$, and hence there is a well defined surjective group homomorphism $\mathrm{GL}_2(\mathcal{O}/z^i\mathcal{O})\cong\mathrm{GL}_2(\mathcal{O})/H_i\rightarrow\mathrm{PGL}_2(\mathcal{O})/\bar{H}_i\cong\mathrm{PGL}_2(\mathcal{O}/z^i\mathcal{O})$. This proves that the group $\mathrm{PGL}_2(\mathcal{O}/z^i\mathcal{O})$ is a quotient of $\mathrm{GL}_2(\mathcal{O}/z^i\mathcal{O})$. It then follows from our general considerations that $K^{(\mathrm{PGL}_2\times{\mathbb{C}^*})_i\rtimes\mathbb{C}^*}(\mathcal{R}_{\leq\lambda}^d)$ is identified with a subspace of $K^{(\mathrm{GL}_2\times{\mathbb{C}^*})_i\rtimes\mathbb{C}^*}(\mathcal{R}_{\leq\lambda}^d)$. Taking a limit, we see that $K^{(\mathrm{PGL}_2\times\mathbb{C}^*)_{\mathcal{O}}\rtimes\mathbb{C}^*}(\mathcal{R}_{\mathrm{PGL}_2,N})$ is identified with a subspace of $K^{(\mathrm{GL}_2\times\mathbb{C}^*)_\mathcal{O}\rtimes\mathbb{C}^*}(\mathcal{R}_{\mathrm{PGL}_2,N})$ as desired.

Now an element of $K^{(\mathrm{GL}_2\times\mathbb{C}^*)_\mathcal{O}\rtimes\mathbb{C}^*}(\mathcal{R}_{\mathrm{PGL}_2,N})$ lies in the subspace $K^{(\mathrm{PGL}_2\times\mathbb{C})_\mathcal{O}\rtimes\mathbb{C}^*}(\mathcal{R}_{\mathrm{PGL}_2,N})$ if and only if its image under the embedding of Lemma~\ref{lem:PGL2spaceoftriples} lies in the quotient of 
\[
\bigoplus_{\mu\in X_*(T)}K^{\bar{T}\times\mathbb{C}^*\times\mathbb{C}^*}(\mathrm{pt})\subset\mathcal{D}_{q,z}
\]
by $(D_1D_2-1)$. Here we write $\bar{T}$ for the diagonal subgroup of $\mathrm{PGL}_2$. Recall that $K^{T}(\mathrm{pt})\cong\mathbb{C}[w_1^{\pm1},w_2^{\pm1}]$ where $w_1$ and $w_2$ correspond, respectively, to the characters $\diag(t_1,t_2)\mapsto t_1$ and $\diag(t_1,t_2)\mapsto t_2$ of~$T$. Similarly, we have $K^{\bar{T}}(\mathrm{pt})\cong\mathbb{C}[X^*(\bar{T})]\cong\mathbb{C}[w_1^{\pm1},w_2^{\pm1}]^{\mathbb{C}^*}$ where $\mathbb{C}^*$ acts by simultaneously rescaling~$w_1$ and~$w_2$. This implies the lemma.
\end{proof}

By Lemma~\ref{lem:GL2PGL2}, we have an isomorphism 
\begin{equation}
\label{eqn:Hamiltonianreduction}
K^{(\mathrm{PGL}_2\times\mathbb{C}^*)_{\mathcal{O}}\rtimes\mathbb{C}^*}(\mathcal{R}_{\mathrm{PGL}_2,N})\cong K^{(\mathrm{GL}_2\times\mathbb{C}^*)_{\mathcal{O}}\rtimes\mathbb{C}^*}(\mathcal{R}_{\mathrm{GL}_2,N})^{\mathbb{C}^*}/(D_1D_2-1)
\end{equation}
where we regard $K^{(\mathrm{GL}_2\times\mathbb{C}^*)_{\mathcal{O}}\rtimes\mathbb{C}^*}(\mathcal{R}_{\mathrm{GL}_2,N})$ as a subalgebra of~$\mathcal{D}_{q,z}$ via the embedding~\eqref{eqn:embedding}. This result can be seen as a dual $K$-theoretic version of Proposition~3.18 of~\cite{BFN18}, which concerns Coulomb~branches defined in terms of equivariant Borel--Moore homology. Note that the decomposition of $\Gr_{\mathrm{GL}_2}$ into connected components leads to a $\mathbb{Z}$-grading and hence a $\mathbb{C}^*$-action on $K^{(\mathrm{GL}_2\times\mathbb{C}^*)_{\mathcal{O}}\rtimes\mathbb{C}^*}(\mathcal{R}_{\mathrm{GL}_2,N})$, but this is not the $\mathbb{C}^*$-action we consider here.

\subsection{Monopole operators}
\label{sec:MonopoleOperators}

Finally, let us describe the Coulomb branch appearing in our main result. In the following we will take $N=\Hom_{\mathbb{C}}(\mathbb{C}^2,\mathbb{C}^2)$ to be the space of $2\times2$ matrices with complex coefficients and let $\mathrm{GL}_2$ act on this space by conjugation. We consider the algebra~$\mathcal{D}_{q,z}$ defined above and the embedding~\eqref{eqn:embedding} of the quantized Coulomb branch into~$\mathcal{D}_{q,z}$.

\begin{definition}[\cite{BFN19,FT19}]
\label{def:monopole}
For each integer $n=1,2$ and each symmetric Laurent~polynomial~$f$ in $n$~variables, we define the elements $E_n[f]$,~$F_n[f]\in\mathcal{D}_{q,t}$, called \emph{dressed minuscule monopole operators}, by the formulas 
\[
E_n[f] = \sum_{\substack{I\subset\{1,2\} \\ |I|=n}}f(w_I)\prod_{r\in I,s\not\in I}\frac{1-qzw_rw_s^{-1}}{1-w_sw_r^{-1}}\prod_{r\in I} D_r
\]
and 
\[
F_n[f] = \sum_{\substack{I\subset\{1,2\} \\ |I|=n}}f(q^{-2}w_I)\prod_{r\in I,s\not\in I}\frac{1-qzw_sw_r^{-1}}{1-w_rw_s^{-1}}\prod_{r\in I} D_r^{-1}.
\]
Here we write $f(w_I)$ (respectively, $f(q^{-2}w_I)$) for the elements of $\mathcal{D}_{q,z}$ obtained by substituting $(w_i)_{i\in I}$ (respectively, $(q^{-2}w_i)_{i\in I}$) into the argument of~$f$.
\end{definition}

Let us write $\mathcal{A}_{q,z}$ for the $\mathbb{C}[q^{\pm1},z^{\pm1}]$-subalgebra of~$\mathcal{D}_{q,z}$ generated by all dressed minuscule monopole operators $E_n[f]$ and $F_n[f]$, together with all symmetric Laurent~polynomials in the variables~$w_i$. The reason for introducing this algebra~$\mathcal{A}_{q,z}$ is that it provides a concrete description of a quantized Coulomb branch.

\begin{proposition}[\cite{AS24}, Proposition~3.5]
\label{prop:monopolegenerate}
The embedding~\eqref{eqn:embedding} provides a $\mathbb{C}$-algebra isomorphism 
\[
K^{(\mathrm{GL}_2\times\mathbb{C}^*)_{\mathcal{O}}\rtimes\mathbb{C}^*}(\mathcal{R}_{\mathrm{GL}_2,N})\cong\mathcal{A}_{q,z}.
\]
\end{proposition}

The conjugation action of $\mathrm{GL}_2$ on~$N=\Hom_{\mathbb{C}}(\mathbb{C}^2,\mathbb{C}^2)$ descends to an action of~$\mathrm{PGL}_2$, so we can consider the quantized Coulomb branch associated to the gauge group~$\mathrm{PGL}_2$ and representation~$N$. Let $\mathbb{C}^*$ act on~$\mathcal{D}_{q,z}$ by simultaneously rescaling the variables~$w_1$ and~$w_2$. The following is immediate from Proposition~\ref{prop:monopolegenerate} the isomorphism~\eqref{eqn:Hamiltonianreduction}.

\begin{proposition}
\label{prop:PGL2monopole}
There is a $\mathbb{C}$-algebra isomorphism 
\[
K^{(\mathrm{PGL}_2\times\mathbb{C}^*)_{\mathcal{O}}\rtimes\mathbb{C}^*}(\mathcal{R}_{\mathrm{PGL}_2,N})\cong\mathcal{A}_{q,z}^{\mathbb{C}^*}/(D_1D_2-1).
\]
\end{proposition}

\section{Proof of the main result}

In this section, we prove that the quantized Coulomb branch described above is isomorphic to a $\mathbb{Z}_2$-invariant subalgebra of the relative skein algebra of a once-punctured torus.

\subsection{Algebraic preliminaries}

Let us write $\widetilde{\mathcal{D}}_{q,z}^0$ for the $\mathbb{C}[q^{\pm1},z^{\pm1}]$-algebra generated by variables $w_1^{\frac{1}{2}}$, $w_2^{\frac{1}{2}}$, $D_1$, $D_2$, and their inverses, subject to the commutation relations 
\[
[D_1,D_2]=[w_1^{\frac{1}{2}},w_2^{\frac{1}{2}}]=0, \quad D_rw_s^{\frac{1}{2}}=q^{\delta_{rs}}w_s^{\frac{1}{2}}D_r.
\]
Let us write $\widetilde{\mathcal{D}}_{q,z}$ for the localization of $\widetilde{\mathcal{D}}_{q,z}^0$ at the multiplicative subset generated by $w_1-q^kw_2$ for $k\in\mathbb{Z}$ and $1-q^k$ for $k\in\mathbb{Z}\setminus\{0\}$. Note that the algebra $\mathcal{D}_{q,z}$ considered above is a subalgebra of~$\widetilde{\mathcal{D}}_{q,z}$. There is a $\mathbb{C}^*$-action on~$\widetilde{\mathcal{D}}_{q,z}$ given by simultaneously rescaling the variables~$w_i^{\frac{1}{2}}$.

\begin{lemma}
\label{lem:relateambientrings}
Define the operator $\varpi\in\End\mathbb{C}_{q,t}(X)$ by the rule $(\varpi f)(X)=f(qX)$ for $f\in\mathbb{C}_{q,t}(X)$. Then the assignments 
\[
q \mapsto q, \quad z\mapsto q^{-1}t, \quad w_1^{\frac{1}{2}}w_2^{-\frac{1}{2}}\mapsto X, \quad D_1\mapsto -q^{-1}X^{-2}\varpi, \quad D_2\mapsto -q^{-1}X^2\varpi^{-1}
\]
determine a well defined injective $\mathbb{C}$-algebra homomorphism $\widetilde{\mathcal{D}}_{q,z}^{\mathbb{C}^*}/(D_1D_2-1)\hookrightarrow\End\mathbb{C}_{q,t}(X)$.
\end{lemma}

\begin{proof}
Any element $f\in\widetilde{\mathcal{D}}_{q,z}$ can be written 
\[
f=\sum_{l_1,l_2\in\mathbb{Z}}f_{l_1,l_2}D_1^{l_1}D_2^{l_2}
\]
where coefficient $f_{l_1,l_2}$ has the form 
\[
\frac{\sum a_{i_1,i_2}w_1^{\frac{i_1}{2}}w_2^{\frac{i_2}{2}}}{(w_1-q^{k_1}w_2)^{j_1}\dots(w_1-q^{k_n}w_2)^{j_n}}=\frac{\sum a_{i_1,i_2}w_1^{\frac{i_1}{2}-j_1-\dots-j_n}w_2^{\frac{i_2}{2}}}{(1-q^{k_1}\frac{w_2}{w_1})^{j_1}\dots(1-q^{k_n}\frac{w_2}{w_1})^{j_n}}
\]
for some $a_{i_1,i_2}\in\mathbb{C}(q,z)$ with the sums running over all $i_1$,~$i_2\in\mathbb{Z}$. If this coefficient is $\mathbb{C}^*$-invariant then numerator of the last expression must be $\mathbb{C}^*$-invariant, that is, $\frac{i_1}{2}+\frac{i_2}{2}-j_1-\dots-j_n=0$. It follows that if $f$ is $\mathbb{C}^*$-invariant then each coefficient $f_{l_1,l_2}$ can be written as a rational expression in $q$,~$z$, and~$w_1^{\frac{1}{2}}w_2^{-\frac{1}{2}}$. By applying the above rules, we can thus map $f$ to an operator in $\End\mathbb{C}_{q,t}(X)$. These assignments map the commutation relations in~$\widetilde{\mathcal{D}}_{q,z}^{\mathbb{C}^*}$ to relations in $\End\mathbb{C}_{q,t}(X)$ and therefore provide a well defined $\mathbb{C}$-algebra homomorphism $\widetilde{\mathcal{D}}_{q,z}^{\mathbb{C}^*}\rightarrow\End\mathbb{C}_{q,t}(X)$. Moreover, the element $D_1D_2-1$ maps to zero, so we get a $\mathbb{C}$-algebra homomorphism $\widetilde{\mathcal{D}}_{q,z}^{\mathbb{C}^*}/(D_1D_2-1)\rightarrow\End\mathbb{C}_{q,t}(X)$.

It remains to show that this map is injective. If $\bar{f}$ lies in the kernel of this map, then we can represent $\bar{f}$ by an expression $f$ which is a Laurent polynomial in~$D_1$ whose coefficients are rational functions in $q$,~$z$, and~$w_1^{\frac{1}{2}}w_2^{-\frac{1}{2}}$. The images of these elements in $\End\mathbb{C}_{q,t}(X)$ satisfy no relations, so we must have $f=0$. This completes the proof.
\end{proof}

We now write $\widetilde{\mathcal{A}}_{q,z}$ for the $\mathbb{C}[q^{\pm1},z^{\pm1}]$-subalgebra of~$\widetilde{\mathcal{D}}_{q,z}$ generated by all dressed minuscule monopole operators $E_n[f]$ and $F_n[f]$, together with all symmetric Laurent polynomials in the variables~$w_i^{\frac{1}{2}}$. Note that the algebra $\mathcal{A}_{q,z}$ considered above is a subalgebra of~$\widetilde{\mathcal{A}}_{q,z}$. The action of~$\mathbb{C}^*$ on~$\widetilde{\mathcal{D}}_{q,z}$ restricts to an action on~$\widetilde{\mathcal{A}}_{q,z}$, and we have the following.

\begin{lemma}
\label{lem:invariantgenerators}
$\widetilde{\mathcal{A}}_{q,z}^{\mathbb{C}^*}$ is generated as a $\mathbb{C}[q^{\pm1},z^{\pm1}]$-algebra by the element $w_1^{\frac{1}{2}}w_2^{-\frac{1}{2}}+w_1^{-\frac{1}{2}}w_2^{\frac{1}{2}}$ together with all elements of the form 
\begin{equation}
\label{eqn:invariantgenerators}
w_1^{-\frac{k}{2}}w_2^{-\frac{k}{2}}E_n[f], \quad w_1^{-\frac{k}{2}}w_2^{-\frac{k}{2}}F_n[f]
\end{equation}
where $n=1,2$ is an integer, $f$ is a homogeneous symmetric polynomial in $n$ variables, and $k$ is the total degree of~$f$.
\end{lemma}

\begin{proof}
It follows from the definitions of $E_n[f]$ and $F_n[f]$ that the above elements are all contained in~$\widetilde{\mathcal{A}}_{q,z}^{\mathbb{C}^*}$. Let $\mathcal{S}$ be the $\mathbb{C}[q^{\pm1},z^{\pm1}]$-algebra consisting of symmetric Laurent polynomials in the~$w_i^{\frac{1}{2}}$. Then the elements~\eqref{eqn:invariantgenerators} generate~$\widetilde{\mathcal{A}}_{q,z}$ as an $\mathcal{S}$-algebra. Therefore, we can write any element $P\in\widetilde{\mathcal{A}}_{q,z}^{\mathbb{C}^*}$ as a finite sum of the form  $P=\sum_ic_iM_i$ where $c_i\in\mathcal{S}$ and $M_i$ is a monomial in the expressions~\eqref{eqn:invariantgenerators} for every index~$i$. We may assume that the monomials $M_i$ appearing in this expression are linearly independent over~$\mathcal{S}$. Since each $M_i$ is $\mathbb{C}^*$-invariant, it then follows that the coefficients $c_i$ must be $\mathbb{C}^*$-invariant. Hence each $c_i$ is a symmetric Laurent polynomial of degree zero in the~$w_i^{\frac{1}{2}}$. The $\mathbb{C}[q^{\pm1},z^{\pm1}]$-algebra of such polynomials is generated by $w_1^{\frac{1}{2}}w_2^{-\frac{1}{2}}+w_1^{-\frac{1}{2}}w_2^{\frac{1}{2}}$, so we are done.
\end{proof}

\begin{lemma}
\label{lem:Coulombrep}
There is an injective $\mathbb{C}$-algebra homomorphism $\widetilde{\mathcal{A}}_{q,z}^{\mathbb{C}^*}/(D_1D_2-1)\hookrightarrow\End\mathbb{C}_{q,t}(X)$ mapping $q\mapsto q$, $z\mapsto q^{-1}t$, and 
\begin{align*}
w_1^{\frac{1}{2}}w_2^{-\frac{1}{2}}+w_1^{-\frac{1}{2}}w_2^{\frac{1}{2}} &\mapsto X+X^{-1}, \\
w_1^{-\frac{k}{2}}w_2^{-\frac{k}{2}}E_1[x^k] &\mapsto q^{-1}t^{\frac{1}{2}}\left(X^kV(X)\varpi+X^{-k}V(X^{-1})\varpi^{-1}\right), \\
w_1^{-\frac{k}{2}}w_2^{-\frac{k}{2}}F_1[x^k] &\mapsto q^{-2k-1}t^{\frac{1}{2}}\left(X^{-k}V(X)\varpi+X^kV(X^{-1})\varpi^{-1}\right),
\end{align*}
where $V$ is defined by the formula in Proposition~\ref{prop:generatingoperators}. This homomorphism maps generators of the form $w_1^{-\frac{k}{2}}w_2^{-\frac{k}{2}}E_2[f]$ and $w_1^{-\frac{k}{2}}w_2^{-\frac{k}{2}}F_2[f]$ to polynomials in $X+X^{-1}$ with coefficients in~$\mathbb{C}_{q,t}$.
\end{lemma}

\begin{proof}
The homomorphism defined in Lemma~\ref{lem:relateambientrings} restricts to an injective $\mathbb{C}$-algebra homomorphism $\widetilde{\mathcal{A}}_{q,z}^{\mathbb{C}^*}/(D_1D_2-1)\hookrightarrow\End\mathbb{C}_{q,t}(X)$. Consider a homogeneous polynomial $f(x_1,x_2)=x_1^ax_2^b+x_1^bx_2^a$ with total degree $k=a+b$. Using the definitions of $E_2[f]$ and $F_2[f]$, one can check that the homomorphism maps 
\begin{align*}
w_1^{-\frac{k}{2}}w_2^{-\frac{k}{2}}E_2[f] &\mapsto X^{a-b}+X^{b-a}=T_{a-b}(X+X^{-1}), \\
w_1^{-\frac{k}{2}}w_2^{-\frac{k}{2}}F_2[f] &\mapsto q^{-2k}(X^{a-b}+X^{b-a})=q^{-2k}T_{a-b}(X+X^{-1}),
\end{align*}
where $T_n$ denotes the Chebyshev polynomial of the first kind of degree~$n$. A general homogeneous symmetric Laurent polynomial of degree~$k$ in the variables $x_1$ and~$x_2$ is a linear combination of polynomials of the form $x_1^ax_2^b+x_1^bx_2^a$. This proves the last statement of the lemma. The remaining statements follow by straightforward calculation using the definitions of the dressed minuscule monopole operators and the rules in Lemma~\ref{lem:relateambientrings}.
\end{proof}

\subsection{The skein algebra from monopole operators}

Using the embedding of Lemma~\ref{lem:Coulombrep}, we can now give two realizations of the relative skein algebra of a once-punctured torus in terms of the dressed minuscule monopole operators.

\begin{proposition}
\label{prop:skeinmonopole1}
There is a $\mathbb{C}$-algebra isomorphism 
\[
\Sk_{A,\lambda}(S_{1,1})\stackrel{\sim}{\longrightarrow}\widetilde{\mathcal{A}}_{q,z}^{\mathbb{C}^*}/(D_1D_2-1)
\]
mapping $A\mapsto q^{-\frac{1}{2}}$, $\lambda\mapsto z$, and 
\[
\alpha \mapsto w_1^{\frac{1}{2}}w_2^{-\frac{1}{2}}+w_1^{-\frac{1}{2}}w_2^{\frac{1}{2}}, \quad
\beta \mapsto q^{\frac{1}{2}}z^{-\frac{1}{2}}F_1[1], \quad
\gamma \mapsto q^2z^{-\frac{1}{2}}w_1^{-\frac{1}{2}}w_2^{-\frac{1}{2}}F_1[x].
\]
\end{proposition}

\begin{proof}
Let us write $\mathcal{M}$ for the $\mathbb{C}_{q,t}$-subalgebra of $\End\mathbb{C}_{q,t}(X)$ generated by $X+X^{-1}$ and the family of operators defined by the formula in Lemma~\ref{lem:curveoperators}. Then Lemma~\ref{lem:Coulombrep} gives an embedding $\widetilde{\mathcal{A}}_{q,z}^{\mathbb{C}^*}/(D_1D_2-1)\hookrightarrow\mathcal{M}$. Since the operators in $\mathcal{M}$ preserve $\mathbb{C}[X^{\pm1}]^{\mathbb{Z}_2}$, we can in fact regard this as an embedding $\Psi:\widetilde{\mathcal{A}}_{q,z}^{\mathbb{C}^*}/(D_1D_2-1)\hookrightarrow\End\mathbb{C}_{q,t}[X^{\pm1}]^{\mathbb{Z}_2}$. We also have the embedding $\Phi_+=\Phi:\Sk_{A,\lambda}(S_{1,1})\hookrightarrow\End\mathbb{C}_{q,t}[X^{\pm1}]^{\mathbb{Z}_2}$ of Proposition~\ref{prop:generatingoperators}. It follows from Lemma~\ref{lem:curveoperators} and Lemma~\ref{lem:Coulombrep} that the images of these embeddings coincide. In particular, we have 
\[
\Phi_+(\alpha)=\Psi(w_1^{\frac{1}{2}}w_2^{-\frac{1}{2}}+w_1^{-\frac{1}{2}}w_2^{\frac{1}{2}}), \quad
\Phi_+(\beta)=\Psi(q^{\frac{1}{2}}z^{-\frac{1}{2}}F_1[1]), \quad
\Phi_+(\gamma)=\Psi(q^2z^{-\frac{1}{2}}w_1^{-\frac{1}{2}}w_2^{-\frac{1}{2}}F_1[x]).
\]
This completes the proof.
\end{proof}

\begin{proposition}
\label{prop:skeinmonopole2}
There is a $\mathbb{C}$-algebra isomorphism 
\[
\Sk_{A,\lambda}(S_{1,1})\stackrel{\sim}{\longrightarrow}\widetilde{\mathcal{A}}_{q,z}^{\mathbb{C}^*}/(D_1D_2-1)
\]
mapping $A\mapsto q^{-\frac{1}{2}}$, $\lambda\mapsto z$, and 
\[
\alpha \mapsto w_1^{\frac{1}{2}}w_2^{-\frac{1}{2}}+w_1^{-\frac{1}{2}}w_2^{\frac{1}{2}}, \quad
\beta \mapsto q^{-1}z^{-\frac{1}{2}}w_1^{\frac{1}{2}}w_2^{\frac{1}{2}}F_1[x^{-1}], \quad
\gamma \mapsto q^{\frac{1}{2}}z^{-\frac{1}{2}}F_1[1].
\]
\end{proposition}

\begin{proof}
Let $\Psi:\widetilde{\mathcal{A}}_{q,z}^{\mathbb{C}^*}/(D_1D_2-1)\hookrightarrow\End\mathbb{C}_{q,t}[X^{\pm1}]^{\mathbb{Z}_2}$ and $\Phi_+:\Sk_{A,\lambda}(S_{1,1})\hookrightarrow\End\mathbb{C}_{q,t}[X^{\pm1}]^{\mathbb{Z}_2}$ be the embeddings appearing in the proof of Proposition~\ref{prop:skeinmonopole1}, and define $\Phi_-\coloneqq\Phi_+\circ\tau_+$.  It follows from Lemma~\ref{lem:curveoperators} and Lemma~\ref{lem:Coulombrep} that the images of these embeddings coincide. In particular, one can check that 
\[
\Phi_-(\alpha)=\Psi(w_1^{\frac{1}{2}}w_2^{-\frac{1}{2}}+w_1^{-\frac{1}{2}}w_2^{\frac{1}{2}}), \quad
\Phi_-(\beta)=\Psi(q^{-1}z^{-\frac{1}{2}}w_1^{\frac{1}{2}}w_2^{\frac{1}{2}}F_1[x^{-1}]), \quad
\Phi_-(\gamma)=\Psi(q^{\frac{1}{2}}z^{-\frac{1}{2}}F_1[1]).
\]
This completes the proof.
\end{proof}

An interesting fact which we do not need here is that the algebra $\widetilde{\mathcal{A}}_{q,z}^{\mathbb{C}^*}/(D_1D_2-1)$ is isomorphic to $K^{(\mathrm{SL}_2\times\mathbb{C}^*)_{\mathcal{O}}\rtimes\mathbb{C}^*}(\mathcal{R}_{\mathrm{PGL}_2,N})$. Thus Propositions~\ref{prop:skeinmonopole1} and~\ref{prop:skeinmonopole2} give $K$-theoretic realizations of the full relative skein algebra of~$S_{1,1}$, and not just its $\mathbb{Z}_2$-invariant subalgebras. In particular, it follows from Section~\ref{sec:SL2Zsymmetries} that there is an $\mathrm{SL}_2(\mathbb{Z})$-action on $K^{(\mathrm{SL}_2\times\mathbb{C}^*)_{\mathcal{O}}\rtimes\mathbb{C}^*}(\mathcal{R}_{\mathrm{PGL}_2,N})$.

\subsection{Invariant subalgebras}

Finally, we come to our main result. In the following, we write~$\xi_1$ and~$\xi_3$ for the involutions of $\Sk_{A,\lambda}(S_{1,1})$ defined in Section~\ref{sec:StatementOfTheMainResult} and write $N$ for the $\mathrm{PGL_2}$-representation considered in Section~\ref{sec:MonopoleOperators}.

\begin{theorem}
\label{thm:main}
There are $\mathbb{C}$-algebra isomorphisms 
\[
\Sk_{A,\lambda}(S_{1,1})^{\xi_1}\cong K^{(\mathrm{PGL}_2\times\mathbb{C}^*)_{\mathcal{O}}\rtimes\mathbb{C}^*}(\mathcal{R}_{\mathrm{PGL}_2,N})\cong\Sk_{A,\lambda}(S_{1,1})^{\xi_3}.
\]
\end{theorem}

\begin{proof}
As in the proof of Lemma~\ref{lem:relateambientrings}, we can write any element of $\widetilde{\mathcal{D}}_{q,z}^{\mathbb{C}^*}$ as a Laurent polynomial in the variables~$D_r$ whose coefficients are rational expressions in~$q$,~$z$, and~$w_1^{\frac{1}{2}}w_2^{-\frac{1}{2}}$. Consider the involution of this algebra that fixes $q,$~$z$, and the~$D_r$ and maps $w_1^{\frac{1}{2}}w_2^{-\frac{1}{2}}\mapsto-w_1^{\frac{1}{2}}w_2^{-\frac{1}{2}}$. It induces an involution of the quotient $\widetilde{\mathcal{D}}_{q,t}^{\mathbb{C}^*}/(D_1D_2-1)$ and its subalgebra $\widetilde{\mathcal{A}}_{q,z}^{\mathbb{C}^*}/(D_1D_2-1)$. Let us write~$\epsilon$ for the involution of the latter. The isomorphism in Proposition~\ref{prop:skeinmonopole1} intertwines $\xi_1$ and $\epsilon$, so we have 
\[
\Sk_{A,\lambda}(S_{1,1})^{\xi_1}\cong\left(\widetilde{\mathcal{A}}_{q,z}^{\mathbb{C}^*}/(D_1D_2-1)\right)^{\epsilon}\cong\mathcal{A}_{q,z}^{\mathbb{C}^*}/(D_1D_2-1).
\]
Similarly, the isomorphism in Proposition~\ref{prop:skeinmonopole2} intertwines $\xi_3$ and~$\epsilon$, so we have 
\[
\Sk_{A,\lambda}(S_{1,1})^{\xi_3}\cong\mathcal{A}_{q,z}^{\mathbb{C}^*}/(D_1D_2-1).
\]
The theorem now follows from Proposition~\ref{prop:PGL2monopole}. 
\end{proof}

In~\cite{AS24} we formulated a general conjecture relating the algebra $\Sk_{A,\lambda}(S)$ for a surface $S=S_{g,n}$ of genus $g\leq1$ with $n>0$ punctures to a quantized $K$-theoretic Coulomb branch. The gauge group used to define this quantized Coulomb branch is isomorphic to $\mathrm{SL}_2^{3g-3+n}$. In view of Theorem~\ref{thm:main}, we expect that $\Sk_{A,\lambda}(S)$ has other dual realizations in terms of quantized Coulomb branches with possibly nonisomorphic gauge groups. Understanding this precisely is an interesting problem, which we leave for future work.

\bibliographystyle{amsplain}

\begin{thebibliography}{99}

\bibitem{AST13} Aharony, O., Seiberg, N., and Tachikawa, Y. (2013). Reading between the lines of four-dimensional gauge theories. \emph{Journal of High Energy Physics}, \textbf{2013}(8), 1--42.

\bibitem{AS24} Allegretti, D.G.L. and Shan, P. (2024). Skein algebras and quantized Coulomb branches. \texttt{arXiv:2401.06737 [math.RT]}.

\bibitem{BFN18} Braverman, A., Finkelberg, M., and Nakajima, H. (2018). Towards a mathematical definition of Coulomb branches of 3-dimensional $\mathcal{N}=4$ gauge theories, II. \emph{Advances in Theoretical and Mathematical Physics}, \textbf{22}(5), 1071--1147.

\bibitem{BFN19} Braverman, A., Finkelberg, M., and Nakajima, H. (2019). Coulomb branches of 3d $\mathcal{N}=4$ quiver gauge theories and slices in the affine Grassmannian. \emph{Advances in Theoretical and Mathematical Physics}, \textbf{23}(1), 75--166.

\bibitem{C05} Cherednik, I. (2005). \emph{Double affine Hecke algebras}. Cambridge University Press.

\bibitem{CG97} Chriss, N. and Ginzburg, V. (1997). \emph{Representation theory and complex geometry}. Birkh\"auser.

\bibitem{DW96} Donagi, R. and Witten, E. (1996). Supersymmetric Yang--Mills theory and integrable systems. \emph{Nuclear Physics~B}, \textbf{460}(2), 299--334.

\bibitem{FT19} Finkelberg, M. and Tsymbaliuk, A. (2019). Multiplicative slices, relativistic Toda and shifted quantum affine algebras. In \emph{Representations and nilpotent orbits of Lie algebraic systems}, pp. 133--304. Birkh\"auser.

\bibitem{F96} Fricke, R. (1896). \"Uber die Theorie der automorphen Modulgruppen. \emph{Nachrichten von der Gesellschaft der Wissenschaften zu G\"ottingen, Mathematisch-Physikalische Klasse}, \textbf{1896}, 91--101.

\bibitem{G12} Gaiotto, D. (2012). $N=2$ dualities. \emph{Journal of High Energy Physics}, \textbf{2012}(8), 1--58.

\bibitem{GMN13a} Gaiotto, D., Moore, G.W., and Neitzke, A. (2013). Wall-crossing, Hitchin systems, and the WKB approximation. \emph{Advances in Mathematics}, \textbf{234}(2013), 239--403.

\bibitem{GMN13b} Gaiotto, D., Moore, G.W., and Neitzke, A. (2013). Framed BPS states. \emph{Advances in Theoretical and Mathematical Physics}, \textbf{17}(2), 241--397.

\bibitem{G09} Goldman, W. (2009). Trace coordinates on Fricke spaces of some simple hyperbolic surfaces. In Handbook of Teichm\"uller theory II, \emph{IRMA Lectures in Mathematics and Theoretical Physics}, \textbf{13}, pp.~611--684.

\bibitem{GKNPS23} Gukov, S., Koroteev, P., Nawata, S., Pei, D., and Saberi, I. (2023). \emph{Branes and DAHA representations}. SpringerBriefs in Mathematical Physics (Vol.~48). Springer Nature.

\bibitem{VV10} Varagnolo, M. and Vasserot, E. (2010). Double affine Hecke algebras and affine flag manifolds, I. In~\emph{Affine flag manifolds and principal bundles}, pp.~233--289.

\bibitem{V89} Vogt, H. (1889). Sur les invariants fondamentaux des \'equations diff\'erentielles lin\'eaires du second ordre. \emph{Annales scientifiques de l'\'Ecole Normale Sup\'erieure}, \textbf{6}, 3--71.

\end{thebibliography}

\end{document}